\documentclass{amsart}

\usepackage{amsmath, amssymb, amsthm}
\usepackage[latin1]{inputenc}
\usepackage{geometry}
\usepackage[normalem]{ulem}
\usepackage{cancel}
\usepackage[bookmarks=false,colorlinks=true,linkcolor=black,citecolor=black,filecolor=black,urlcolor=black]{hyperref}
\usepackage{tikz}

\newtheorem{theorem}{Theorem}[section]
\newtheorem{lemma}[theorem]{Lemma}
\newtheorem{proposition}[theorem]{Proposition}

\newtheorem{example}[theorem]{Example}

\newtheorem{corollary}[theorem]{Corollary}

\newcommand{\Gal}{{\rm Gal}}

\newcommand{\Aut}{\mbox{\rm Aut}}
\newcommand{\Inn}{\mbox{\rm Inn}}

\newcommand{\End}{\mbox{\rm End}}

\newcommand{\inv}{^{-1}}

\newcommand{\Z}{{\mathbb Z}}
\newcommand{\Q}{{\mathbb Q}}

\newcommand{\C}{{\mathbb C}}

\newcommand{\F}{{\mathbb F}}

\newcommand{\GL}{{\rm GL}}

\newcommand{\PSL}{{\rm PSL}}

\newcommand{\tr}{{\rm tr}}

\newcommand{\sgn}{{\rm sgn}}

\newcommand{\Nr}{{\rm Nr}}
\newcommand{\Tr}{{\rm Tr}}
\newcommand{\Exp}{{\rm Exp}}

\newcommand{\matriz}[1]{\begin{array} #1 \end{array}}
\newcommand{\pmatriz}[1]{\left(\begin{array} #1 \end{array}\right)}
\newcommand{\GEN}[1]{\langle #1 \rangle}

\newcommand{\A}{{\mathcal A}}

\newcommand{\U}{\mathcal{U}}

\newcommand{\qand}{\quad \text{and} \quad}

\newcommand{\Supp}{{\rm Supp}}

\newcommand{\diag}{\operatorname{diag}}
\newcommand{\Cl}{\operatorname{Cl}}
\newcommand{\Hom}{\operatorname{Hom}}
\newcommand{\Aug}{\operatorname{Aug}}
\newcommand{\aug}{\operatorname{aug}}
\newcommand{\Irr}{\operatorname{Irr}}

\newcommand{\pa}[2]{\varepsilon_{#2}\left(#1\right)}
\newenvironment{proofof}{\par\bigskip\noindent{\bf Proof of }}{\qed\par\bigskip}

\date{\today}

\title{Finite groups in integral group rings}

\thanks{Partially supported by Ministerio de Economía y Competitividad project MTM2012-35240 and Fondos FEDER and Fundación Séneca of Murcia 04555/GERM/06.
\\
The first version of these notes was prepared for The School of Advances in Group Theory and Applications (AGTA) celebrated in Vietri sul Mare (Salerno, Italy) in June of 2016.
They were revised to be presented at the Vrije Universiteit of Brussels (VUB) to the students of the master course ``Non-commutative Algebra'' of the academic year 2017-18 while the author held the  ``VUB Leerstoel''.
It was revised again for two mini courses at the IME of the Universidade de São Paulo in September of 2019 and in the conference Group algebras, representations and computation at ICTS in Bangalore in October 2019, and at Algebra Seminar of the University of Murcia in November 2022.
The author wants to express his gratitude to the AGTA Group, the VUB, the USP and the ICTS for their invitation, support and hospitality.}

\author{Ángel del Río}
\address{Ángel del Río, Departamento de Matemáticas, Universidad de Murcia,
30100, Murcia, Spain}
\email{adelrio@um.es}

\begin{document}
\begin{abstract}
We revise some problems on the study of finite subgroups of the group of units of integral group rings of finite groups and some techniques to attack them.
\end{abstract}

\maketitle

The study of the group of units $\U(\Z G)$ of the integral group ring of a finite group $G$ was started by Higman in \cite{Higman1940Thesis} (see also \cite{Higman1940Paper}) and has been an active subject of research since.
Two basics references for this topic are the book of Sehgal \cite{Sehgal1993} and the two volumes book by Jespers and the author \cite{JespersdelRioGRG1,JespersdelRioGRG2}.
The aim of this note is to introduce the reader to the investigation of the finite subgroups of $\U(\Z G)$ and, in particular, of the torsion units in $\Z G$.
For a more advanced and updated treatment of the topic see \cite{MargolisdelRioAGTA}. 

\section{Basic notation}

All throughout $G$ is a finite group, denoted multiplicatively, and $Z(G)$ denotes the center of $G$. The order of a set $X$ is denoted $|X|$. We also use $|g|$ to denote the order of a torsion group element $g$.

Every ring $R$ is assumed to have an identity and its center, Jacobson radican and group of units are denoted $Z(R)$, $J(R)$ and $\U(R)$, respectively.
If $n$ is a positive integer then $M_n(R)$ denotes the ring of $n\times n$ matrices with entries in $R$ and $\GL_n(R)=\U(M_n(R))$, the group of units of $M_n(R)$.
If $M$ is an $R$-module then $\End_R(M)$ denotes the ring of endomorphisms of $M$ and $\Aut_R(M)$ denotes the group of automorphisms of $M$.
If $M$ is free of rank $n$ then there is a natural isomorphism $\End_R(M)\rightarrow M_n(R)$ associating every homomorphism with its expression in a fixed basis, which restricts to a group isomorphism $\Aut_k(M)\rightarrow \GL_n(R)$.
We will use these isomorphisms freely to identify endomorphisms and matrices.

The group ring of $G$ with coefficients in $R$ is denoted $RG$. It contains $R$ as a subring an its group of units contains $G$ as a subgroup which is also a basis of $RG$ as a left $R$-module. Moreover the elements of $R$ and $G$ commute. The group ring is characterized by the following property, which we refer as the \emph{Universal Property of Group Rings}: For every ring homomorphism $f:R\rightarrow S$ and every group homomorphism $\phi:G\rightarrow \U(S)$ such that $f(r)\phi(g)=\phi(g)f(r)$ for every $r\in R$ and every $g\in G$ there is a unique ring homomorphism $f'$ extending $f$ and $\phi$. In particular, if $S$ is a ring containing $R$ as subring then every group homomorphism $\phi:G\rightarrow \U(S)$ with image commuting with the elements of $R$ extends uniquely to a ring homomorphism $RG\rightarrow S$, which we will also denote $\phi$. 

We will abuse slightly the notation so that any time that we write $r=\sum_{g\in G} r_gg\in RG$ we are implicitly assuming that each $r_g$ belong to $R$. The \emph{support} of $r$ is 
	$$\Supp(r)= \{g \in G : r_g\ne 0\}.$$

\section{The Berman-Higman Theorem}

We start with a very useful result with many consequences on the finite subgroups of $\U(\Z G)$.

\begin{theorem}[Berman-Higman Theorem] \cite{Berman1955,Higman1940Thesis}\label{BermanHigman}
If $u=\sum_{g\in G} u_g g$ is a torsion unit of $\Z G$ then either $u=\pm 1$ or $u_1=0$.
\end{theorem}

\begin{proof}
The key observation is that every complex invertible matrix of finite order is diagonalizable.
This is a consequence of the fact that an elementary Jordan matrix
  $$J_k(a) = \pmatriz{{ccccc} a & & & \\ 1 & a &  & \\ & \ddots & \ddots & \\ & & 1 & a &  \\ & & & 1 & a} \in M_k(\C).$$
is of finite order if and only if $k=1$ and $a$ is a root of unity.

Consider the regular representation, i.e. the group homomorphism $G\rightarrow \End_{\C}(\C G)$ associating $g\in G$ with the map $\rho(g):x\mapsto gx$.
Representing $\rho_g$ in the basis $G$, we deduce that if $n=|G|$ then the trace of $\rho(1)$ is $n$ and if $g\in G\setminus \{1\}$, then the trace of $\rho(g)$ is $0$.
Identifying $\End_{\C}(\C G)$ and $M_n(\C)$ we have a group homomorphism $\rho:G\rightarrow \U(M_n(\C))=\GL_n(\C)$.
By the Universal Property of Group Rings, $\rho$ extends to a $\C$-algebra homomorphism
$\rho:\C G\rightarrow  M_n(\C)$.

Suppose that $u=\sum_{g\in G} u_gg$ is a torsion unit of $\Z G$, say of order $m$.
Then $\rho(u)$ is  diagonalizable, so it is conjugate in $M_n(\C)$ to a diagonal matrix
$\diag(\xi_1,\dots,\xi_n)$, where each $\xi_i$ is a complex $m$-th root of unity.
As the trace map $\tr:M_n(\C)\rightarrow \C$ is $\C$-linear, we have
  $$n u_1 = \sum_{g\in G} u_g \tr(\rho(g)) = \tr(\rho(u)) = \tr(\diag(\xi_1,\dots,\xi_n)) = \sum_{i=1}^n \xi_i.$$
Taking absolute values we have
  $$n|u_1| \le \sum_{i=1}^n |\xi_i|=n$$
and equality holds if and only if all the $\xi_i$'s are equal. Thus, if not all the $\xi_i$'s are equal then $u_1$ is an integer with absolute value less
than $1$, i.e. $u_1=0$. Otherwise $\diag(\xi_1,\dots,\xi_n)=\xi I$, where $I$ denotes the identity matrix. As $\xi I$ is central we have $\rho(u)=\xi I$ and
$u_1=\xi$, an integral root of unity. Thus, $\xi=\pm 1$ and $\rho(u)=\pm I=\rho(\pm 1)$. As $\rho$ is injective on $\C G$, we deduce that $u=\pm 1$.
\end{proof}

The most obvious torsion units of $\Z G$ are the elements of the form $\pm g$ with $g\in G$.
They are called \emph{trivial units} of $\Z G$.

As a consequence of the Berman-Higman Theorem (Theorem~\ref{BermanHigman}), one can describe all the torsion central units.

\begin{corollary}\label{TorsionCentral}
The torsion central units of $\Z G$ are the trivial units $\pm g$ with $g\in Z(G)$. In particular, if $G$ is abelian then every finite subgroup of $\U(\Z G)$ is
contained in $\pm G$.
\end{corollary}

\begin{proof}
Let $u$ be a torsion central unit of $\Z G$ and let $g\in \Supp(u)$.
Then $v=ug^{-1}$ is a torsion unit with $1\in \Supp(v)$.
By Theorem~\ref{BermanHigman},  $v=\pm 1$, and so $u=\pm g$.
\end{proof}

The proof of Theorem~\ref{BermanHigman} uses one of the main tools in the study of group rings, namely Representation Theory.
Let $R$ be a commutative ring and let $M$ be a left $RG$-module.
The map associating $g\in G$ to the $R$-endomorphism of $M$ given by $m\mapsto gm$ is a group homomorphism $G\mapsto \Aut_R(M)$. 
Conversely, if $M$ is an $R$-module then, by the Universal Property of Group Rings, every group homomorphism $G\rightarrow \Aut_R(M)$ extends to a ring homomorphism $RG\rightarrow \End_R(M)$ and this induces a structure of $RG$-module on $M$. 
Thus we can identify $RG$-modules with group homomorphism $G\rightarrow \End_R(M)$ with $M$ an $R$-module.

An $R$-\emph{representation} of $G$ of \emph{degree} $k$ is a group homomorphism $\rho:G\rightarrow \GL_k(R)$.
Our identification of $\End_R(R^k)$ and $M_k(R)$ allows to identify $\rho$ with the $RG$-module whose underlying $R$-module is $R^k$ and $gm=\rho(g)m$ for $g\in G$ and $m\in R^k$.
The composition of $\rho$ with the trace map $\tr:M_k(R)\rightarrow R$ is called \emph{the character} afforded by $\rho$, or by the underlying $RG$-module. 
Observe that both $\rho$ and the character afforded by $\rho$ are $R$-linear maps defined not only on $G$ but also on $RG$.

For example, the trivial map $G\rightarrow \U(R), g\mapsto 1$ is a character of degree $1$ and its linear span to $RG$ is called the \emph{augmentation map}:
  \begin{eqnarray*}
   \aug_G:RG & \rightarrow & R \\ \sum_{g\in G} r_g g & \mapsto & \sum_{g\in G} r_g.
  \end{eqnarray*}
The kernel $\Aug(RG)$ of $\aug_G$ is called the \emph{augmentation ideal} of $RG$.
As the augmentation map is a ring homomorphism it restricts to a group homomorphism
  $$\aug_G:\U(RG)  \rightarrow \U(R).$$
The kernel of this group homomorphism is denoted $V(RG)$, i.e.
	$$V(RG)=\{u\in \U(RG) : \aug_G(u)=1\}.$$
The elements of $V(RG)$ are usually called \emph{normalized units}.
If $R$ is commutative then $\U(RG)=\U(R)\times V(RG)$.
In particular, $\U(\Z G)=\pm V(\Z G)$ and hence the study $\U(\Z G)$ and $V(\Z G)$ are equivalent.

More generally, if $N$ is a normal subgroup of $G$ then the natural map $G\rightarrow G/N \subseteq \U(R (G/N))$ is an $R(G/N)$-representation of $G$ which
extends linearly to a ring homomorphism
  \begin{eqnarray*}
   \aug_{G,N}:RG & \rightarrow & R(G/N) \\ \sum_{g\in G} r_g g & \mapsto & \sum_{g\in G} r_g gN.
  \end{eqnarray*}
We set $\Aug_N(RG)=\ker(\aug_{G,N})$. The reader should prove:
  \begin{equation}\label{AugSubgroup}
  \Aug_N(RG) = RG \Aug(RN)= \Aug(RN)RG, \qand \Aug(RG) = \bigoplus_{g\in G\setminus \{1\}} R(g-1).
  \end{equation}
Observe that $\aug_G=\aug_{G,G}$ and hence $\Aug(RG)=\Aug_G(RG)$.
Moreover, if $N_1\subseteq N_2$ are normal subgroups of $G$ then $\aug_{N_2} = \Phi\circ \aug_{G/N_1,N_2/N_1} \circ \aug_{G,N_1}$, where $\Phi$ is the $R$-linear extension of the natural isomorphism $\frac{G/N_1}{N_2/N_1}\cong G/N_2$.
Hence $\Aug_{N_1}(RG)\subseteq \Aug_{N_2}(RG)$.
Furthermore, $\aug_{G,1}=1_{RG}$ and so $\Aug_1(RG)=0$.

If $N$ is a normal subgroup of $G$ then we also set
  $$V(RG,N) = \{u\in \U(RG) : \aug_{G,N}(u)=1.\}$$
Observe that $V(RG,G)=V(RG)$, $V(RG,1)=1$ and if $N_1\subseteq N_2$ are normal subgroup of $G$ then $V(RG,N_1)\subseteq V(RG,N_2)$.

One of the main questions on group rings is the so called Isomorphism Problem:

\bigskip
\noindent (ISO-$R$): \textbf{The Isomorphism Problem for group rings over a ring $R$}: 
\begin{center}
Does $RG\cong RH$ imply $G\cong H$?
\end{center}
\bigskip

(ISO) is an abbreviation of (ISO-$\Z$) and called the \textbf{Isomorphism Problem}. 
Observe that $RG \cong R\otimes_{\Z} \Z G$ and therefore if $\Z G\cong \Z H$ then $RG\cong RH$ for every ring $R$. Thus a negative solution for (ISO) is a negative solution for (ISO-$R$) for every ring $R$. 
More generaly, if there is a ring homomorphism $R\rightarrow S$ then we can see $S$ as an $R$-module and $S\otimes_R RG\cong SG$. Thus a positive solution for (ISO-$S$) implies a positive solution for (ISO-$R$).

It is easy to find negative solutions for the Isomorphism Problem for $R=\C$. For example, if $G$ is an abelian group then $\C G\cong \C^{|G|}$ and therefore if $H$ is another abelian group with $|G|=|H|$ then $\C G\cong \C H$.

Suppose that $G$ and $H$ are finite groups and let $f:\Z G \rightarrow \Z H$ be a ring homomorphism.
Then $f'(g)=\aug(f(g))f(g)$ is a group homomorphism and hence it extends to a ring homomorphism $f':\Z G \rightarrow \Z H$ such that $f'(G)\subseteq V(\Z H)$.
This shows that if $\Z G$ and $\Z H$ are isomorphic then there is an isomorphism $f:\Z G\rightarrow \Z H$ such that $f(G)$ is a subgroup of $V(\Z H)$ with the
same order as $H$.

\begin{corollary}
	The Isomorphism Problem holds for finite abelian groups.
\end{corollary}

\begin{proof}
Let $G$ and $H$ be finite groups and suppose that $G$ is abelian and suppose that $\Z G$ and $\Z H$ are isomorphic. Then necessarily $H$ is abelian (why?).
By the remark prior to the corollary, there is an isomorphism $f:\Z G\rightarrow \Z H$ which maps $V(\Z G)$ onto $V(\Z H)$. Moreover, by Corollary~\ref{TorsionCentral}, the set of torsion units of $V(\Z G)$ and $V(\Z H)$ are $G$ and $H$, respectively. 
Then $f$ restricts to an isomorphism $f:G\rightarrow H$.
\end{proof}

Another consequence of the Berman-Higman Theorem is the following:

\begin{corollary}\label{FiniteSubgroupIndependent}
Every finite subgroup of $V(\Z G)$ is linearly independent over $\Q$ (equivalently, over $\Z$).
\end{corollary}
\begin{proof}
	Let $H=\{u_1,\dots,u_n\}$ be a finite subgroup of $V(\Z G)$ and suppose that
	$$c_1u_1+\dots+c_nu_n = 0$$
	with $c_i\in \Z$. Then
	$$c_1+c_2u_2u_1^{-1}+\dots+c_nu_nu_1^{-1}=0$$
	and each $u_iu_1^{-1}$, with $i=2,\dots,n$ is a torsion element of $V(\Z G)\setminus \{1\}$. By the Berman-Higman Theorem (Theorem~\ref{BermanHigman}),
	$1\not\in \Supp(u_iu_1^{-1})$ for every $i\ne 1$ and therefore, comparing the coefficients of $1$ in both sides of the previous equality, we deduce that $c_1=0$. This shows that each $c_i=0$.
\end{proof}

An obvious consequence of Corollary~\ref{FiniteSubgroupIndependent} is that if $H$ is a finite subgroup of $V(\Z G)$ then the subring $\Z[H]$ of $\Z G$ generated by $H$ is isomorphic to the group ring $\Z H$. 
Clearly $H$ is a basis of $\Z[H]$ over $\Z$.
Furthermore, if $|H|=|G|$ then $H$ is a basis of $\Q G$ over $\Q$.
Actually, by the following lemma, it is also a basis of $\Z G$ over $\Z$

\begin{corollary}\label{GroupBasis}
The following are equivalent for a finite subgroup $H$ of $V(\Z G)$:
\begin{enumerate}
	\item $|H|=|G|$.
	\item $\Z G=\Z[H]$.
	\item $H$ is an basis of $\Z G$ over $\Z$.
\end{enumerate}
\end{corollary}

\begin{proof}
(3) implies (2) and (2) implies (1) are obvious. Suppose that $|H|=|G|$.
Clearly $\Z[H]\subseteq \Z G$.
 and we have just observed that $\Q G=\Q[H]$. Thus $n\Z G\subseteq \Z[H]$ for some positive integer $n$.
So, if $g\in G$ then $ng=\sum_{h\in H} m_hh$ for some $m_h\in \Z$. Thus, for every $h\in H$ we have $ngh^{-1} = m_h+\sum_{k\in H\setminus \{h\}} m_k kh^{-1}$.
Applying once more the Berman-Higman Theorem we deduce that the coefficient of $1$ in $\sum_{k\in H\setminus \{h\}} m_k kh^{-1}$ is $0$. Therefore $m_h=na$ where $a$ is the coefficient of $1$ in $gh^{-1}$.
Thus $m_h$ is a multiple of $n$ for every $h$ and hence $g=\sum_{h\in H} \frac{m_h}{n} h \in  \Z[H]$.
This proves that $\Z G=\Z[H]$ and hence $H$ is an integral basis of $\Z G$.
\end{proof}

Observe that, by Corollary~\ref{GroupBasis}, the Isomorphism Problem can be restated as whether all the subgroups of $V(\Z G)$ with the same cardinality as $G$ are isomorphic.

As every idempotent matrix with entries in a field $F$ of characteristic $0$, is diagonalizable with eigenvalues $0$ and $1$, if $\rho$ is an $F$-representation of $G$ and $e$ is an idempotent of $FG$ then $\rho(e)$ is conjugate to a diagonal matrix with entries $0$ and $1$. Moreover the number of ones in the diagonal is the rank of $\rho(e)$. Therefore $\chi(e)$ is the rank of $\rho(e)$.
Using thid snf the same idea as for the proof of the Berman-Higman Theorem one can obtain the following:

\begin{lemma}\label{TraceIdempotent}
	Let $K$ be a field extension of $\Q$ and let $e=\sum_{g\in G} e_g g\in KG$ with $e^2=e\not\in \{0,1\}$. Then $e_1$ is a rational number in the interval $(0,1)$.
\end{lemma}

\begin{proof}
	Let $\rho$ be the regular representation of $G$ and $\chi$ the character afforded by $\rho$. Then all the eigenvalues of $\rho(e)$ are $0$ or $1$ and
	$\chi(e)$ is the multiplicity of $1$ as eigenvalue of $\rho(e)$. As $e\not\in \{0,1\}$ and $\rho$ is injective, $\chi(e)\in \{1,\dots,|G|-1\}$ and
	$e_1=\frac{\chi(e)}{|G|}$.
\end{proof}

\begin{corollary}\label{FiniteSubgroupDividesG}
	The order of every finite subgroup of $V(\Z G)$ divides $|G|$.
\end{corollary}
\begin{proof}
	Let $\rho$ be the regular representation and let $\chi$ be the character afforded by $\rho$.
	
	Let $H$ be a finite subgroup of $V(\Z G)$ and let $e=\widehat{H}=\frac{\sum_{h\in H} h}{|H|}$. Then $e$ is an idempotent of $\Q G$ and hence $r=\chi(e)$, where $r$
	is the rank of $\rho(e)$. On the other hand $\chi(h)=|G|c_h$ where $c_h$ is the coefficient of $1$ in $h$. By the Berman-Higman Theorem, $c_h=0$ unless $h=1$.
	Therefore $r=\chi(e)=\frac{|G|}{|H|}$, is an integer and thus $|H|$ divides $|G|$.
\end{proof}

\section{Problems on finite subgroups of $\U(\Z G)$}\label{SectionProblems}

In this section we collect some of the main problems on the finite groups of units of $\Z G$.
The results of the previous sections suggests that there is a strong connection between the finite subgroups $H$ of $V(\Z G)$ and the subgroups of $G$.
For example, the elements of $H$ are linearly independent  over $\Q$ (Corollary~\ref{FiniteSubgroupIndependent}) and the order of $H$ divides the order of $G$ (Corollary~\ref{FiniteSubgroupDividesG}).
Moreover, if $G$ is abelian then the torsion elements of $V(\Z G)$ are just the elements of $G$ (Corollary~\ref{TorsionCentral}).
We cannot expect that the latter generalizes to non-abelian groups because conjugates of $G$ in $\U(\Z G)$ are not included in $G$, in general.
So the most that we can expect is that the finite subgroups of $V(\Z G)$ are conjugate to subgroups of $G$ or at least isomorphic to subgroups of $G$.

\begin{example}{\rm
		Consider $S_3$, the symmetric group on three symbols which we realized as the semidirect product $S_3=\GEN{a}_3\rtimes \GEN{b}_2$, with $a^b=a^{-1}$.
		The ordinary character table of $S_3$ is as follows:
		$$\matriz{{c|rrr}
			& 1 & a & b \\\hline
			\aug & 1 & 1 & 1\\
			\sgn & 1 & 1 & -1\\
			\chi & 2 & -1 & 0
		}$$
		Moreover, $\chi$ is afforded by the following representation:
		$$\rho(a)=\pmatriz{{cc} -2 & -3 \\ 1 & 1}, \quad \rho(b)=\pmatriz{{cc} 1 & 0 \\ -1 & -1}.$$
		(This is not the most natural representation affording $\chi$, but it is well adapted to our purposes.)
		Therefore the map $\phi:\Q S_3 \rightarrow \Q \times \Q \times M_2(\Q), x \rightarrow (\aug(x),\sgn(x),\rho(x))$ is an algebra isomorphism.
		In particular $\phi$ restricts to an isomorphism from $\Z S_3$ to $\phi(\Z S_3)$ and the latter can be easily calculated using integral Gaussian elimination because it is the additive subgroup generated by the image of $S_3$ by $\phi$. After some straightforward calculations we have that
		$$\phi(\Z S_3) = \left\{ \left(x,y,\pmatriz{{cc} a & 3b \\ c & d} \right): x,y,a,b,c,d\in \Z,
		\matriz{{c} x\equiv y \bmod 2, \\ x \equiv a \bmod 3, \\ y \equiv d \bmod 3} \right\}.$$
		For example, there is $u\in \Z S_3$ with $\phi(u)=(1,-1,\diag(1,-1))$. As $\phi(u)^2=(1,1,I_2)$, $u$ is an element of order $2$ in $V(\Z
		S_3)$.
		The projection of $\phi(b)$ and $\phi(u)$ in the third coordinate are $\rho(b)=\pmatriz{{cc} 1 & 0 \\ 1 & 1}$ and $\rho(u)=\diag(1,-1)$, respectively.
		Although the diagonal form of $\rho(b)$ is $\rho(u)$, any invertible matrix $U\in M_2(\C)$ with $U^{-1}\rho(b)U=\rho(u)$ is of the form
		$$U=\pmatriz{{cc} 2x & 0 \\ x & y}$$
		with $x,y\in \C\setminus \{0\}$.
		Therefore if $U\in M_2(\Z)$, i.e. $x,y\in \Z$, then
		$$U^{-1}=\pmatriz{{cc} \frac{1}{2x} & 0 \\ \frac{1}{y} & \frac{1}{y}}\not\in M_2(\Z).$$
		This proves that $\rho(b)$ and $\rho(u)$ are not conjugate in $M_2(\Z)$ and therefore $\phi(b)$ and $\phi(u)$ are not conjugate in $\phi(\Z S_3)$. Since $\phi$ is injective $b$ and $u$ is not conjugate in $\Z S_3$ to $b$. As the three involutions of $S_3$ are conjugate in $S_3$. It follows that $u$ is not conjugate in $\Z S_3$ to any element of $S_3$.
		As all the elements of order $2$ of $S_3$ are conjugate in $S_3$, we conclude that $u$ is not conjugate in $\U(\Z G)$ to any element of $G$.
		However, $\rho(u)$ and $\rho(b)$ are conjugate in $M_2(\Q)$ and thus $\phi(u)$ and $\phi(b)$ are conjugate in $\phi(\Q G)$.
		As $\phi$ is an isomorphism, $u$ and $b$ are conjugate in the units of $\Q G$.
}\end{example}

The previous example shows that not all the torsion elements of $V(\Z G)$ are conjugate to elements of $S_3$ in $\Z S_3$. However, using the isomorphism $\phi$ it can be easily proven that the torsion element of $V(\Z S_3)$ are conjugate in $\Q S_3$ to an element of $S_3$. This suggests the following problems where we use the following terminology: two subgroups or elements of $\U(\Z G)$ are said to be \emph{rationally conjugate} if they are conjugate within the units of $\Q G$.

\newpage
\noindent\textbf{The Zassenhaus Problems}\footnote{These problems have been known for a long time as the Zassenhaus Conjectures because, at least an affirmative to (ZC1) was mentioned as a conjecture by H. Zassenhaus \cite{Zassenhaus1974}. S.K. Sehgal attributed to Zassenhaus the three as conjectures in \cite{Sehgal1978} and even if negative solutions to (ZC2) and (ZC3) are are known since the beginning of the 1990s the authors kept calling them the Zassenhaus Conjectures. Since we also know now counterexamples for the first one, I prefer to call them problems now.}:
Given a finite group $G$:
\begin{itemize}
	\item[(ZC1)] Is every torsion element of $V(\Z G)$ rationally conjugate to an element of $G$?
	\item[(ZC2)] Is every finite subgroup of $V(\Z G)$, with the same order as $G$, rationally conjugate to $G$?
	\item[(ZC3)] Is every finite subgroup of $V(\Z G)$ rationally conjugate to a subgroup of $G$?
\end{itemize}
\bigskip

For brevity we say that one of the problems (ISO), (ZC1), ... holds for a group when it has a positive answer and otherwise we say that it does not hold. More generaly (P) implies (Q) means that if (P) holds for a group $G$ then so does (Q).

Clearly (ZC3) implies (ZC1) and (ZC2).
Moreover (ZC2) implies (ISO), or more precisely if (ZC2) holds for a finite group $G$ and $\Z G\cong \Z H$ for another group $H$ then $G\cong H$.

The following proposition shows that in the Zassenhaus Problems one can replace $\Q$ by any field of characteristic $0$. For its proof we need some notation.

If $F$ is a field, $A$ is a finite dimensional $F$-algebra and $a\in A$ then the \emph{norm} of $a$ over $F$ is $\Nr_{A/F}(a)=\det(\rho(a))$ where $\rho:A\rightarrow \End_F(A)$ is the regular representation of $A$, i.e. $\rho(a)(b)=ab$, for $a,b\in A$.  Observe that if $B$ is a basis of $A$ over $F$ then $\Nr_{A/F}(a) = \det(\rho_B(a))$, where $\rho_B(a)$ is the matrix representation of $\rho(a)$ in the basis $B$.
Let $P(X)=\chi_{A/F}(a)$, the characteristic polynomial of $\rho_B(a)$ over $F$. By the Cayley-Hamilton Theorem, $P(a)=0$. 
Furthermore, the independent term of $P(X)$ is equal to $\pm \Nr_{A/F}(a)$.
Therefore, $0=P(a)=\pm \Nr_{A/F}(a)+aQ(a)$ for some $Q\in F(X)$. 
Then $a\in \U(A)$ if and only if $\Nr_{A/F}(a)\ne 0$.
Moreover, if $E$ is a field containing $F$ as a subfield then $B$ is also a basis of $E\otimes_F A$ over $E$ and hence, considering $A$ embedded in $E\otimes_F A$ via the map $a\mapsto 1\otimes a$, we have $\Nr_{A\otimes_F E/E}(a) = \det(\rho_B(a)) = \Nr_{A/F}(a)$ for every $a\in A$.

\begin{proposition}\label{ConjugationIndependentField}
	Let $E/F$ be an extension of infinite fields, let $A$ be a finite dimensional $F$-algebra and let $B=E\otimes_F A$.
	Let $M$ and $N$ be finite subsets of $A$ which are conjugate within $B$. Then they are also conjugate within $A$.
\end{proposition}

\begin{proof}
	Fix an $F$-basis $\{b_1,\dots,b_d\}$ of $A$.
	Let $u$ be a unit of $B$ such that $M^u=N$. For every $m\in M$ let $n_m=u^{-1}mu$. So the system of equations $Xn_m=mX$ has a solution in the units of $B$.
	Expressing this in terms of the $F$-basis $b_1,\dots,b_d$ of $A$ we obtain a system of homogeneous linear equations in $d$ unknowns, with coefficients in
	$F$ which has a solution $(e_1,\dots,e_d)$ in $E$ such that $e_1b_1+\dots+e_nb_d$ is a unit of $B$.
	Let $v_1,\dots,v_k$ be an $F$-basis of the set of solutions and consider the polynomial
	$f(X_1,\dots,X_k)=\Nr_{A/F}(X_1v_1+\dots+X_kv_k)=\Nr_{B/E}(X_1v_1+\dots+X_kv_k)$.
	By elementary linear algebra $v_1,\dots,v_k$ is also an $E$-basis of the set of solutions in $E$.
	Thus $e_1b_1+\dots+e_kb_k=x_1v_1+\dots+x_kv_k$ for some $x_1,\dots,x_k\in E$ and hence $f(x_1,\dots,x_k)\ne 0$.
	This implies that $f$ is not the zero polynomial. Then $f(y_1,\dots,y_k)\ne 0$ for some $y_1,\dots,y_k\in F$, since $F$ is infinite.
	Therefore $v=y_1v_1+\dots+y_kv_k$ is an element of $A$ with $\Nr_{A/F}(v)\ne 0$ and $vn_m=mv$ for each $m\in M$.
	The first implies that $v\in \U(A)$ and the second that $M^v=N$. We conclude that $M$ and $N$ are conjugate within $A$.
\end{proof}

Applying Proposition~\ref{ConjugationIndependentField} to $A=\Q G$ and $F$ a field containing $\Q$, and having in mind that $FG \cong F\otimes_{\Q} G$, we get the following:

\begin{corollary}\label{ConjugadoNoImportaCuerpo}
	Let $H$ be a finite subgroup of $V(\Z G)$ and let $F$ be a field containing $\Q$ then $H$ is rationally conjugate to a subgroup of $G$ if and only if it is conjugate in $FG$ to a subgroup of $G$.
\end{corollary}

\begin{corollary}\label{ConjugateCharacters}
	Let $H_1$ and $H_2$ be subgroups of $\U(\Z G)$. Then $H_1$ and $H_2$ are rationally conjugate if and only if there is an isomorphism $\phi:H_1\rightarrow H_2$
	such that $\chi(h)=\chi(\phi(h))$ for every $h\in H_1$ and every $\chi\in \Irr(G)$.
\end{corollary}
\begin{proof}
	The necessary condition is obvious.
	Suppose that $\phi:H_1\rightarrow H_2$ is an isomorphism satisfying the condition.
	For every $\chi\in \Irr(G)$ fix a representation $\rho_{\chi}$ affording $\chi$.
	Then $\Phi=(\rho_{\chi})_{\chi\in \Irr}:\C G\rightarrow \prod_{\chi\in  \Irr(G)} M_{\chi(1)}(\C)$ is an isomorphism of $\C$-algebras.
	Moreover $\rho_{\chi}|_{H_1}$ and $\rho_{\chi}|_{H_2} \circ \phi$ are representations of $H_1$ affording the same character, namely $\chi|_{H_1}=\chi|_{H_2} \circ \phi$.
	Thus $\rho_{\chi}|_{H_1}$ and $\rho_{\chi}|_{H_2} \circ \phi$ are equivalent as $\C$-representations, i.e. there is $U_{\chi}\in M_{\chi(1)}(\C)$ such that
	$\rho_{\chi} \phi(h)=U^{-1} \rho_{\chi}(h) U$ for every $h\in H_1$.
	Hence $u=\Phi((U_{\chi})_{\chi\in \Irr(G)})$ is a unit of $\C G$ such that
	$u^{\inv}hu=\phi(h)$ for every $h\in H_1$. Thus $u^{-1}H_1u=\phi(H_1)=H_2$, i.e. $H_1$ and $H_2$ are conjugate in $\C G$. We conclude that $H_1$ and $H_2$ are conjugate in
	$\Q G$, by Corollary~\ref{ConjugadoNoImportaCuerpo}.
\end{proof}

If we replace conjugacy by isomorphism we obtain versions of the Zassenhaus Problems.
For example, the Isomorphism Problem is the ``isomorphism version'' of (ZC2) asking whether all the subgroups of $\Z G$ with the same cardinality as $G$ are isomorphic.
The isomorphism versions of (ZC3) is the following question:

\bigskip
\noindent \textbf{The Subgroup Problem}:
(ISOS) Is every finite subgroup of $V(\Z G)$ isomorphic to a subgroup of $G$?
\bigskip

The isomorphism version of (ZC1) is known as the Spectrum Problem.
The set of orders of the torsion elements of a group $\Gamma$ is call the \emph{spectrum} of $\Gamma$.
Observe that two cyclic groups are isomorphic if and only if they have the same cardinality. Therefore the isomorphism version of (ZC1) is the following:

\bigskip
\noindent\textbf{The Spectrum Problem}:
(SpP) Do $G$ and $V(\Z G)$ have the same spectra?
\bigskip

A weaker version of the Spectrum Problem is the Prime Graph Question which was proposed by Kimmerle. The \emph{prime graph} of $\Gamma$ is the undirected graph whose vertices are the prime integers $p$ with $p=|g|$ for some $g\in \Gamma$ and the edges are the
pairs $\{p,q\}$ of different primes $p$ and $q$ with $pq=|g|$ for some $g\in \Gamma$, i.e. with $pq$ in the spectrum of $G$.

\bigskip
\noindent\textbf{The Prime Graph Question}: (PQ)
Does $G$ and $V(\Z G)$ have the same prime graph?
\bigskip

By the Cohn-Livingstone Theorem (Proposition~\ref{CohnLivingstone}), the spectra of $G$ and $V(\Z G)$ contain the same prime powers.
Moreover, by Proposition~\ref{FiniteSubgroupDividesG} and Sylow Theorem, the sets of orders of the finite $p$-subgroups of $V(\Z G)$ and $G$ coincide.
This suggest the following particular cases of the Subgroup Problem and (ZC2):

\bigskip
\noindent\textbf{The Sylow Subgroup Problem}:
(SyP) Is every finite $p$-subgroup of $V(\Z G)$ isomorphic to a subgroup of $G$?
\bigskip

\bigskip
\noindent\textbf{The Sylow-Zassenhaus Problem}:
(SZP) Is every finite $p$-subgroup of $G$ rationally conjugate to a subgroup of $G$?
\bigskip

A weaker version of the Zassenhaus Problem (ZC1) was proposed by Kimmerle.
Similar weaker versions of (ZC2) and (ZC3) make sense.

\bigskip
\noindent\textbf{The Weak Zassenhaus Problems}:
\begin{itemize}
	\item[(WZP1)] Is every torsion element of $V(\Z G)$ conjugate to an element of $G$ in $\Q H$ for some finite group $H$ containing $G$ as subgroup?
	\item[(WZP2)] Is every finite subgroup of $V(\Z G)$ with the same order as $G$ conjugate to $G$ in $\Q H$ for some finite group $H$ containing $G$ as subgroup?
	\item[(WZP3)] Is every finite subgroup of $V(\Z G)$ conjugate to a subgroup of $G$ in $\Q H$ for some finite group $H$ containing $G$ as subgroup?
	\item[(WSZP)] Is every finite $p$ subgroup of $V(\Z G)$ conjugate to a subgroup of $G$ in $\Q H$ for some finite group $H$ containing $G$ as subgroup?
\end{itemize}
\bigskip

A final question related with these problems is the Automorphism Problem which tries to predicts how the automorphisms of $\Z G$ are.
Consider the following subgroups of $\Aut(\Z G)$:
	$$\Aut_*(\Z G) = \{\alpha \in \Aut(\Z G) : \aug(\alpha(x))=\aug(x) \text{ for all }x\in \Z G\}$$
and 
	$$\Aut_h(\Z G) = \{\alpha \in \Aut(\Z G) : \alpha(g)\in \Z g \text{ for all } 
	g\in G\}.$$
The latter is easy to describe. 
Let $G_0=\Hom(G,\{1,-1\})$, the set formed by the group homomorphism $G\rightarrow \{1,-1\}$, with the group structure given by pointwise multiplication: $(\alpha \beta)(x)=\alpha(x)\beta(x)$. 
Observe that $G_0$ is isomorphic to the group of linear characters of the Sylow $2$-subgroup of $G/G'$ and hence $G_0$ is isomorphic to the Sylow 2-subgroup of $G/G'$.
Every $\beta\in G_0$ determines an element $\overline{\beta}$ of $\Aut_h(\Z G)$ with $\overline{\beta}(g)=\beta(g)g$ for each $g\in G$ and $\beta\mapsto \overline{\beta}$ defines an isomorphism $G_0\rightarrow \Aut_h(\Z G)$. 

Let $\alpha\in \Aut(\Z G)$. Then $f(g)=\aug(\alpha(g))\alpha(g)$ and $\beta(g)=\aug(\alpha(g))$ define group homomorphisms $f:G\rightarrow \U(\Z G)$ and $\beta \in G_0$ and $f(G)$ is a basis of $\Z G$, as $\Z$-module.
Hence $f$ extends to an automorphism of $\Z G$, also denoted $f$, and $f\in \Aut_*(\Z G)$. 
Then
 $\alpha(\sum_{g\in G} a_g g)=\sum_{g\in G} a_g \alpha(g) = \sum_{g\in G} a_g \beta(g) f(g) = f(\sum_{g\in G} a_g \beta(g)g)=
 (f  \circ \overline{\beta})(\sum_{g\in G} a_g g)$. 
This proves that $\Aut(\Z G)=\Aut_*(\Z G)\Aut_h(\Z G)$ and clearly $\Aut_*(\Z G)\cap \Aut_h(\Z G)=1$.

So the description of $\Aut(\Z G)$ reduces to that of $\Aut_*(\Z G)$.
Every automorphism of $G$ extends uniquely to an element of $\Aut_*(\Z G)$.
We can identify the latter with the group $\Aut(G)$ of automorphisms of $G$ so we see $\Aut(G)$ as a subgroup of $\Aut_*(\Z G)$.
Also, the inner automorphisms of $\Z G$ belong to $\Aut_*(\Z G)$.
More generally, the inner automorphisms of $\Q G$ leaving $\Z G$ invariant form another normal subgroup of $\Aut_*(\Z G)$.
We denote this group $\Inn_{\Q G}(\Z G)$. Then $\Aut(G)\Inn_{\Q G}(\Z G)$ is a subgroup of $\Aut_*(\Z G)$.

\bigskip
\noindent \textbf{The Automorphism Problem}
(AUT) Is $\Aut_*(\Z G)=\Aut(G)\Inn_{\Q G}(\Z G)$?
\bigskip

\begin{proposition}\label{ZP2-ISO-AUT}
	(ZC2) holds for $G$ if and only if (ISO) and (AUT) for $G$.
\end{proposition}

\begin{proof}
	Suppose that (ZC2) holds for $G$ and let $H$ be a subgroup of $V(\Z G)$ of cardinality $|G|$.
	Then $H$ is rationally conjugate to $G$ and hence $G\cong H$.
	Thus (ISO) holds for $G$.
	Suppose now that $\alpha\in \Aut_*(\Z G)$. Then $H=\alpha(G)$ is a subgroup of $V(\Z G)$ with the same order as $G$. By assumption, there is a unit $u$ of $\Q G$ such that $H=u^{-1}Gu$. Let $\beta$ be the inner automorphism of $\Q G$ defined by $u$.
	Then $\beta(\Z G) =\Z H\subseteq \Z G$ and therefore $\beta\in \Inn_{\Q G}(\Z G)$.
	Moreover, $\beta^{-1}\alpha \in \Aut(G)$. Thus $\alpha \in \Aut(G)\Inn_{\Q G}(\Z G)$.
	We conclude that (AUT) holds for $G$.
	
	Conversely, suppose that (ISO) and (AUT) hold for $G$.
	Let $H$ be a subgroup of $G$ with the same order as $G$.
	By the Universal Property of Group Rings there is a ring homomorphism $\beta:\Z H\rightarrow \Z G$ whose restriction to $H$ is the identity of $H$.
	As $G$ and $H$ have the same order, $\beta$ is an isomorphism and hence, by assumption, there is an isomorphism $\alpha:G\rightarrow H$.
	Applying again the Universal Property of Group Rings there is a ring isomorphism $\Z G\rightarrow \Z H$ extending $\alpha$, which we also denote $\alpha$.
	Then $\beta\alpha\in \Aut_*(\Z G)$ and by assumption $\beta\alpha = \delta\gamma$ for some $\gamma\in \Aut(G)$ and $\delta\in \Inn_{\Q G}(\Z G)$.
	Then $H=\beta(H) = \delta \gamma \alpha^{-1}(H) = \delta(G)$. Therefore $H$ is rationally conjugate to $G$. This proves that (ZC2) holds for $G$.
\end{proof}

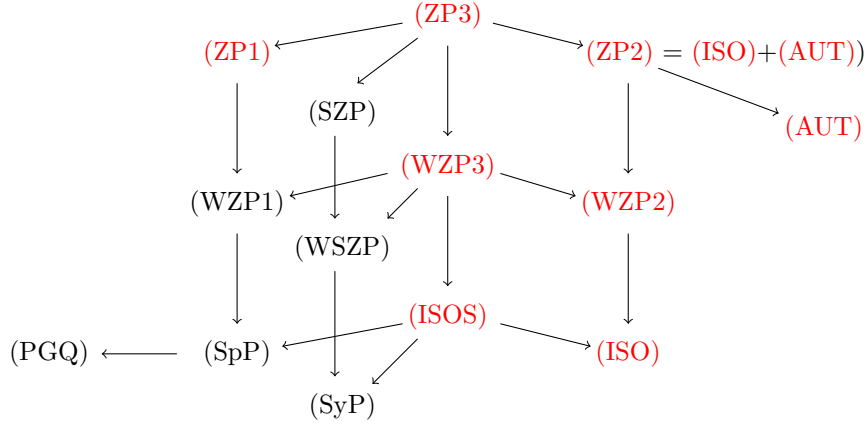
\begin{figure}[h!]
	\begin{center}
		\begin{tikzpicture}
		\node at (2.8,5.5) {\textcolor{red}{(ZC3)}};
		\draw[<-] (0.5,5.1)--(2.2,5.4);
		\node at (0,5) {\textcolor{red}{(ZC1)}};
		\draw[->] (3.4,5.4)--(4.6,5.1);
		\node at (6.5,5) {\textcolor{red}{(ZC2)} = \textcolor{red}{(ISO)}+\textcolor{red}{(AUT)})};
		\draw[->] (5.6,4.8)--(7.2,4.2);
		\node at (7.8,4) {\textcolor{red}{(AUT)}};
		
		\draw[->] (0,4.6)--(0,3.4);
		\draw[->] (2.8,5.1)--(2.8,3.9);
		\draw[->] (5.2,4.6)--(5.2,3.4);
		
		\node at (2.8,3.5) {\textcolor{red}{(WZP3)}};
		\draw[->] (3.5,3.4)--(4.5,3.1);
		\node at (0,3) {(WZP1)};
		\draw[<-] (0.7,3.1)--(2,3.4);
		\node at (5.2,3) {\textcolor{red}{(WZP2)}};		
		\draw[->] (2.4,5.2)--(1.6,4.6);
		\node at (1.4,4.2) {(SZP)};		
		\draw[->] (2.4,3.2)--(2,2.8);
		\node at (1.4,2.4) {(WSZP)};
		\draw[->] (2.4,1.2)--(1.8,0.6);
		\node at (1.4,0.3) {(SyP)};
		\draw[->] (1.3,3.9)--(1.3,2.8);
		\draw[->] (1.3,2.1)--(1.3,0.7);
				
		\draw[->] (0,2.6)--(0,1.4);
		\draw[->] (2.8,3.1)--(2.8,1.9);
		\draw[->] (5.2,2.6)--(5.2,1.4);

		\node at (0,1) {(SpP)};
		\draw[<-] (0.6,1.1)--(2.2,1.4);
		\node at (2.8,1.5) {\textcolor{red}{(ISOS)}};
		\draw[->] (3.5,1.4)--(4.7,1.1);
		\node at (5.2,1) {\textcolor{red}{(ISO)}};
		
		\draw[->] (0,0.6)--(0,-0.4);
		\node at (0,-1) {(PQ)};
		\end{tikzpicture}
	\end{center}
	\caption{\label{ProblemImplications}
		Logical implications between problems on finite subgroups of $V(\Z G)$. 
		Red means that some negative solution is known. }
\end{figure}

%\newpage
Figure~\ref{ProblemImplications} collects the logical implications between the problems introduced in this section.
We list here a few relevant results on them.
See \cite{MargolisdelRioAGTA} for a more extensive list of results. 
We start with negative solutions which justified the red color in Figure~\ref{ProblemImplications}.

\textbf{Negative results}:
\begin{itemize}
	\item Roggenmkamp and Scott  constructed the first counterexample to (AUT) \cite{Roggenkamp1991} and Klingler discovered a simpler one \cite{Klingler1991}.
	This provides negative answers to the Zassenhaus Conjectures stating that (ZC2) and (ZC3) holds true for every group.
	
	\item Hertweck showed a counterexample to (ISO) \cite{Hertweck2001}.
	Of course this is another negative solution for (ZC2) but it is more complicated than the counterexamples of Roggenkamp and Scott and Klingler.
	
	\item Recently Eisele and Margolis \cite{EiseleMargolis18} have proved that a group proposed in \cite{MargolisdelRioCWAlgorithm} is a counterexample to the longest standing conjecture of Zassenhaus, namely the one stating that (ZC1) holds true for all finite groups.
\end{itemize}

\textbf{Positive solutions for (ZC3)}:
(ZC3) holds (and hence all the problems mentioned in this section) for the following groups.
\begin{itemize}
	\item nilpotent groups \cite{Weiss1991}.
	\item split metacyclic groups $A\rtimes X$ with $A$ and $X$ cyclic of coprime order \cite{Valenti1994}. The proof of this result is based in a previous proof in \cite{PolcinoSehgal1984} of a positive solution for (ZC1) for this class of groups.
\end{itemize}

\textbf{Positive solutions for (ZC1)}:
Besides the groups in the previous list positive solutions for (ZC1) has been proved for the following families of groups:
\begin{itemize}
	\item All the groups of order at most 143 \cite{BachleHermanKonovalovMargolisSingh2018}.
	\item groups with a normal Sylow subgroup with abelian complement \cite{Hertweck2006}.
	\item cyclic-by-abelian groups \cite{CaicedoMargolisdelRio2013}.
	\item $\PSL(2,q)$ for $q$ either a Fermat or Mersenne prime or $q \in \{8, 9, 11, 13, 16, 19, 23, 25, 32 \}$ \cite{LutharPassi1989, Hertweck2006, HertweckBrauer, Hertweck2008A6, KimmerleKonovalov2017, BachleMargolis2017, MargolisdelRioSerrano2017}.
%    \item $\SL(2,p)$ and $\SL(2,p^2)$ for $p$ prime \cite{delRioSerrano2018}. 
\end{itemize}

\textbf{Positive solutions for (ISO)}:
Withcomb proved (ISO) for metabelian groups, i.e. groups whose derived subgroup is abelian \cite{WhitcombThesis}.

\textbf{Negative solutions for the Isomorphism Problem for group algebras}:

\begin{itemize}
	\item Dade founded two metabelian groups $G$ and $H$ such that $FG\cong FH$ for every field $F$ \cite{Dade1971}.
	\item García-Lucas, Margolis and del Río obtained recently a negative answer for the \emph{Modular Isomorphism Problem} which is the version of the Isomorphism Problem for coefficients in the field with $p$ elements (in other versions for fields of characteristic $p$) and finite $p$-groups \cite{GarciaMargolisdelRio}. Actually, they only answer in negative the case $p=2$. The case $p>2$ is still open.
\end{itemize}

\section{$p$-elements}

Let $p$ be a prime integer.
Recall that an element of order a power of $p$ in a group is called a \emph{$p$-element}.
In this section we collect some results on $p$-elements of $V(\Z G)$.

We start describing the Jacobson radical of group algebras of $p$-groups over fields of characteristic $p$.

\begin{lemma}\label{JacobsonAugmentation}
	Let $F$ be a field of characteristic $p>0$ and let $G$ be a finite group.
	\begin{enumerate}
		\item If $G$ is a $p$-group then $\Aug(FG)=J(FG)$.
		\item If $P$ is a normal $p$-subgroup of $G$ then $\Aug_P(FG)$ is nilpotent.
	\end{enumerate}
\end{lemma}
\begin{proof}
	(1) Suppose that $|G|=p^n$. As $FG$ is artinian, $\Aug(FG)\subseteq J(FG)$ if and only if $\Aug(FG)$ is nilpotent. As $\dim_F(FG/\Aug(FG))=1$, to prove (1) it is
	enough to show that $\Aug(FG)$ is nilpotent. We argue by induction on $n$. The case $n=1$ is obvious because in this case $FG$ is commutative, $\Aug(FG)$ is
	spanned as vector space over $F$, by the element of the form $g-1$, with $g\in G$ and $(g-1)^p=g^p-1=0$.
	Suppose that $n>1$ and let $H$ be a non-trivial central subgroup of $G$ of order $p$. By induction hypothesis, $\Aug(F(G/H))$ and $\Aug(FH)$ are nilpotent.
	Moreover $\aug_{G,H}(\Aug(FG))=\Aug(F(G/H))$. Therefore, using \eqref{AugSubgroup} we obtain, $\Aug(FG)^m \subseteq \ker \aug_{G,H}=\Aug_H(FG)=FG \Aug(FH)$, for some $m$.
	As $\Aug(FH)$ is nilpotent, so is $\Aug(FG)$.
	
	(2) As $J(FP)$ is nilpotent, so is $\Aug_P(FG)=FG \Aug(FP)=FG J(FP)=J(FP) FG$.
\end{proof}

\begin{lemma}\label{VPPgroup}
	If $p$ is a prime integer and $P$ is a normal $p$-subgroup of $G$ then every torsion element of $V(\Z G,P)$ is a $p$-element.
\end{lemma}
\begin{proof}
	Let $q$ be a prime integer different from $p$, let $u\in V(\Z G,P)$ of order $q$ and let $x=u-1$.
	Let $\F_p$ denote the field with $p $ elements.
	Then $x\in \Aug_P(\Z G)$ and hence the image of $x$ in $\F_pG$ is nilpotent by Lemma~\ref{JacobsonAugmentation}. Thus there is a positive integer $n$ such that $x^{p^n}\equiv 0\mod p\Z G$ and hence $u^{p^n} \equiv 1 \mod p\Z G$.
	As $u^q=1$ and $p$ and $q$ are different primes, we have $u \equiv 1 \mod p\Z G$.
	Thus $x=p^i y$ for some positive integer $i$ and $y\in \Z G$. If $x\ne 0$ then one may assume that $y\not \in p\Z G$.
	Then
	$$0=u^q-1=p^i\left (qy+\binom{q}{2}p^iy^2+\binom{q}{3}p^{2i}y^3+\dots+p^{(q-1)i}y^q\right),$$
	so that $p\mid y$, a contradiction. Thus $x=0$ and hence $u=1$, contradicting our hypothesis that $u$ has order $q$.
\end{proof}

Let $R$ be a ring. Then $[R,R]$ denotes the additive subgroup of $R$ generated by the Lie brackets
$$[x,y]=xy-yx, \quad (x,y\in R).$$
If $S$ is a subring of the center of $R$ then $R\times R \rightarrow R, (x,y)\mapsto [x,y]$ is an $S$-bilinear map.
Therefore $[R,R]$ is an $S$-submodule of $R$.
If moreover, $R=SX$, i.e. $R$ is generated by $X$ as $S$-module then $[R,R]$ is generated by $\{[x,y]:x,y\in X\}$ as $S$-module.
In particular, if $R$ is a commutative ring then $[RG,RG]=\sum_{g,h\in G} R[g,h]$.

\begin{lemma}\label{pthPowerCommutator}
	Let $p$ be a prime integer and let $R$ be an arbitrary ring. Then for every $n$ and $x,y\in RG$ we have
	$$(x+y)^{p^n} \equiv x^{p^n}+y^{p^n} \mod (pR+[R,R]).$$
	Moreover, if $x\in [R,R]$ then $x^p \in pR+[R,R]$.
\end{lemma}
\begin{proof}
	As $pR$ is an ideal of $R$, by factoring modulo $pR$ we may assume that $pR=0$.
	Let $Z$ be the set formed by non-constant $p$-tuples with entries in $\{x,y\}$.
	Then $$(x+y)^p = x^p+y^p + \sum_{(z_1,\dots,z_p)\in Z} z_1z_2\cdots z_p.$$
	Consider the cyclic group $C_p=\GEN{g}$ of order $p$ acting on $Z$ by cyclic permutation, i.e.
		$$g\cdot (z_1,z_2,\dots,z_p)=(z_2,\dots,z_p,z_1).$$
	The orbit $O$ of $(z_1,z_2,\dots,z_p)$ of this action has exactly $p$-elements and, as $pz_1\dots z_p=0$ we have
	\begin{eqnarray*}
		&& z_1z_2\cdots z_p + z_2\cdots z_pz_1 + \dots + z_pz_1\cdots z_{p-1} \\
		&=&  (z_2\cdots z_pz_1-z_1z_2\cdots z_p) + \dots + (z_pz_1\cdots z_{p-1}-z_1z_2\cdots z_p)  \\
		&=& [z_2\cdots z_p,z_1]+[z_3\cdots z_p,z_1z_2]+\dots [z_p,z_1\cdots z_{p-1}] \in [R,R].
	\end{eqnarray*}
	Classifying the products $z_1z_2\dots z_p$ by orbits we deduce that $\sum_{(z_1,\dots,z_p)\in Z} z_1z_2\cdots z_p\in [R,R]$. 	
	This proves that for every $x,y\in R$, $(x+y)^p=x^p+y^p+\alpha$ for some $\alpha\in [R,R]$.
	In particular, there is $\alpha \in [R,R]$ with $[x,y]^p = (xy-yx)^p = (xy)^p-(yx)^p+\alpha=[x,(yx)^{p-1}y]+\alpha\in [R,R]$.
	Using this it easily follows that $\alpha^p\in [R,R]$ for every $\alpha \in [R,R]$.
	Then, arguing by induction on $n$, there are $\alpha,\beta\in [R,R]$ such that
	$$(x+y)^{p^n}=(x^p+y^p+\alpha)^{p^{n-1}} = x^{p^n}+y^{p^n}+\alpha^{p^{n-1}}+\beta \equiv x^{p^n}+y^{p^n} \mod [R,R].$$
\end{proof}

Given $a=\sum_{g\in G} a_g g\in RG$, with $a_g\in R$ for every $g\in G$ and a subset $X$ of $G$ we set
$$\varepsilon_X(a) = \sum_{x\in X} a_x.$$
The Berman-Higman Theorem states that if $u$ is a torsion element of $V(\Z G)$ of order different from one then $\varepsilon_{\{1\}}(x)=0$.
This notation will be used mostly with $X$ a conjugacy class of $G$ and with the sets of the form
$$G[n] = \{g\in G : |g|=n\}.$$
If $g\in G$ then $g^G$ denotes the conjugacy class of $g$ in $G$ and the \emph{partial augmentation} of $a$ at $g$ is  $\varepsilon_{g^G}(a)$.
When the group $G$ is clear from the context we simplify the notation by writing $\varepsilon_g(a)$ rather than $\varepsilon_{g^G}(a)$.

\begin{lemma}\label{CommutatorPartialAugmentation}
	If $R$ is a commutative ring and $G$ is a group then
	$$[RG,RG]= \sum_{g,h\in G} R[g,h] = \{a\in RG : \varepsilon_C(a)=0 , \text{for each conjugacy class } C \text{ of } G\}.$$
\end{lemma}

\begin{proof}
	That the first two sets are equal was already mentioned just before Lemma~\ref{pthPowerCommutator}. That the second is included in the third follows because $\varepsilon_C$ is $R$-linear and $\varepsilon_C([g,h])=0$ for every $g,h\in G$.
	To finish the proof observe that if $a$ belong to the third set then $a$ is a sum of elements of the form $x=\sum_{t\in T} x_t g^t$ with $g\in G$, $T$ a right transversal of $C_G(g)$ in $G$ and $x_t\in R$ with $\sum_{t\in T} x_t=0$.
	For such $x$ we have $x=\sum_{t\in T} x_t g^t-(\sum_{t\in T} x_t) g=\sum_{t\in T} x_t (g^t-g)=\sum_{t\in T} x_t [t^{-1}g,t]\in [RG,RG]$.
	Thus $a$ is a sum of elements in $[RG,RG]$, so that $a\in [RG,RG]$.
\end{proof}

\begin{proposition}[Cohn-Livingstone \cite{CohnLivingstone1965}] \label{CohnLivingstone}
	Let $u$ be a torsion element of $V(\Z G)$ and let $p$ be a prime integer. Then
	$$|u|=p^n \quad \Leftrightarrow \quad \varepsilon_{G[p^n]}(u) \not\equiv 0 \mod p.$$
\end{proposition}
\begin{proof}
	Write $u=\sum_{g\in G} u_g g$.
	By Lemma~\ref{pthPowerCommutator},
	\begin{equation}\label{Potenciap}
	u^{p^n}=\sum_{g\in G} u_g^{p^n}g^{p^n}+x + py.
	\end{equation}
	with $x\in [\Z G,\Z G]$ and $y\in \Z G$.
	By, the Berman-Higman Theorem we have
	$$\varepsilon_1(u^{p^n})=\begin{cases} 1, & \text{if } u^{p^n}=1; \\ 0, & \text{otherwise}\end{cases}.$$
	By \eqref{Potenciap} and Lemma~\ref{CommutatorPartialAugmentation} we have
	$$\varepsilon_1(u^{p^n})  \equiv \sum_{g\in \bigcup_{i=0}^n G[p^i]} u_g^{p^n} \equiv
	\left(\sum_{i=0}^n \varepsilon_{G[p^i]}(u) \right)^{p^n} \equiv
	\sum_{i=0}^n \varepsilon_{G[p^i]}(u) \mod p.$$
	Therefore, if the order of $u$ is $p^n$ then
	$$\sum_{i=0}^b \varepsilon_{G[p^b]}(u) \equiv \epsilon_1(u^{p^b})= \begin{cases} 0 \mod p, & \text{if } b<n; \\ 1 \mod p, & \text{otherwise}.\end{cases}$$
	Thus
	$$\varepsilon_{G[p^b]}(u) \equiv \begin{cases} 1 \mod p, & \text{if } b=n; \\ 0\mod p, & \text{otherwise}.\end{cases}$$
	If the order of $u$ is not a power of $p$ then $\sum_{i=0}^b \varepsilon_{G[p^b]}(u) \equiv 0 \mod p$ for every positive integer $b$ and hence $\varepsilon_{G[p^n]}\equiv 0 \mod p$ for every $n\ge 0$.
\end{proof}

Recall that the \emph{exponent} of $G$, denoted $\Exp(G)$, is the least common multiple of the orders of the elements of $G$, or equivalently the smallest
positive integers $e$ such that $g^e=1$ for every $g\in G$.

\begin{corollary}\label{PrimeSpectrum}
	$V(\Z G)$ and $G$ have the same primary spectrum, i.e. for every prime and every positive integer $G$ contains an element of order $p^n$ if and only if so does $V(\Z G)$.
	In particular, the least common multiple of the orders of the torsion elements of $V(\Z G)$ is the exponent of $G$.
\end{corollary}

Observe that two groups might have the same primary spectrum but not the same spectrum. For example, the spectrum of  $S_3$ is $\{1,2,3\}$ while the spectrum of a cyclic group of order $6$ is $\{1,2,3,6\}$.

\section{Partial augmentations}

In this section we present one of the techniques to attack the problems introduced in Section~\ref{SectionProblems}.

Using Lemma~\ref{CommutatorPartialAugmentation} it easily follows that if $T$ is a set of representatives of the conjugacy classes of $G$ then
$$[RG,RG] = \bigoplus_{t\in T,g\in t^G\setminus \{t\}} R (g-t).$$
Therefore $RG/[RG,RG]$ is a free $R$-module with rank the number of conjugacy classes of $G$.
Moreover, if $S$ is a subring of $R$ then
$$[SG,SG]=SG \cap [RG,RG].$$

\begin{lemma}\label{MRSWL}
	The following conditions are equivalent for a finite subgroup $H$ of $V(\Z G)$.
	\begin{enumerate}
		\item $H$ is rationally conjugate to a subgroup of $G$;
		\item there is a homomorphism $\phi:H\rightarrow G$ such that for every $h\in H$ and every $g\in G\setminus \phi(h)^G$,
		$\varepsilon_g(h)=0$.
		\item there is a homomorphism $\phi:H\rightarrow G$ such that $\varepsilon_g(h)=\varepsilon_g(\phi(h))$ for every $h\in H$ and $g\in G$.
	\end{enumerate}
\end{lemma}
\begin{proof}
	(1) implies (2). Suppose that $u^{-1}Hu\le G$ with $u\in \U(\Q G)$ and consider the group homomorphism $\phi:H\rightarrow G, h\mapsto u^{-1}hu$.
	Then
	$$h-\phi(h) = [hu,u^{-1}] \in \Z G\cap [\Q G,\Q G] = [\Z G,\Z G].$$
	Thus, if $g\in G\setminus \phi(h)^G$ then
	$$0=\varepsilon_g(h-\phi(h))=\varepsilon_g(h).$$
	
	(2) implies (3). Suppose that $\phi:H\rightarrow G$ is a group homomorphism satisfying the condition in (2).
	Then
	$$\varepsilon_g(h) = \begin{cases}
	\aug(h) = 1,& \text{if } g\in \phi(h)^G; \\
	0, & \text{if } g\not\in \phi(h)^G.
	\end{cases}$$
	Thus $\varepsilon_g(h)= \varepsilon_g(\phi(h))$ for every $h\in H$ and $g\in G$, i.e. $\phi$ satisfies (3).
	
	(3) implies (1) Suppose that $\phi:H\rightarrow G$ satisfies condition (3).
	Therefore, $\varepsilon_g(\phi(h)-h)=0$ for each $g\in G$ and hence $\phi(h)-h\in [\Z G,\Z G]$, by Lemma~\ref{CommutatorPartialAugmentation}.
	Moreover, $\phi$ is injective, because if $\phi(h)=1$ then $\varepsilon_1(h)=1$. Thus $h=1$ by the Berman-Higman Theorem.
	Therefore $\phi$ is an isomorphism from $H$ to $\phi(H)$ and the latter is a subgroup of $G$.
	If $\chi\in \Irr(G)$ then $\chi([\Z G,\Z G])=0$ and hence $\chi(h)=\chi(\phi(h))$.
	By Corollary~\ref{ConjugateCharacters}, $H$ and $\phi(H)$ are conjugate in $\Q G$.
\end{proof}

\begin{theorem}[Marciniak-Ritter-Sehgal-Weiss \cite{MarciniakRitterSehgalWeiss1987}] \label{MRSW}
	Let $u$ be an element of order $n$ of $V(\Z G)$. Then the following are equivalent:
	\begin{enumerate}
		\item $u$ is conjugate in $\Q G$ to an element of $G$.
		\item For every $i=1,\dots,n-1$, there is exactly one conjugacy class $C$ of $G$  with $\varepsilon_C(u^i)\ne 0$.
		\item $\varepsilon_C(u^i)\ge 0$, for every $i=1,\dots,n-1$ and every conjugacy class $C$ of $G$.
	\end{enumerate}
\end{theorem}
\begin{proof}
	(1) $\Rightarrow$ (2) is a consequence of Lemma~\ref{MRSWL}. 
	(2) $\Leftrightarrow$ (3) follows easily from the fact that the sum of the partial augmentations $\varepsilon_C(u)$ of $u$ is $\aug(u)=1$.
	
	Suppose that (2) holds.
	For every $i=1,\dots,n$ let $g_i\in G$ such that $\varepsilon_{g^G}(u^i)=0$ for every $g\in G\setminus g_i^G$ other than the one containing $g_i$.
	By the Berman-Higman Theorem $g_i=1$ if and only if $u^i=1$ if and only if $i=n$. 
	By Lemma~\ref{MRSWL}, it is enough to prove that $g_i$ is conjugate to $g_1^i$ in $G$ for every $i=1,\dots,n$, because this implies that $u^i\rightarrow
	g_1^i$ is a group homomorphism with $\varepsilon_g(u^i)=0$ for each $g\in G\setminus (g_1^i)^G$.
	Writing $i$ as a product of primes, and arguing by induction on the number of primes in the factorization of $i$ it is enough to prove this for $i$ prime.
	This will follow at once from the following:
	
	\textbf{Claim}: Let $v\in V(\Z G)$, let $p$ be a prime integer and let $x,y\in G$ such that $\varepsilon_g(v)=0$ for every $g\in G\setminus x^G$ and
	$\varepsilon_g(v^p)=0$ for every $g\in G\setminus y^G$.
	Then $x^p$ and $y$ are conjugate in $G$.
	
	Indeed, as $\varepsilon_g(v)=\varepsilon_g(x)$ and $\varepsilon_g(v^p)=\varepsilon_g(y)$ for each $g\in G$ and $\aug(v)=\aug(v^p)=1$, it follows from Lemma~\ref{CommutatorPartialAugmentation} that $v\equiv x
	\mod [\Z G,\Z G]$ and $v^p \equiv y \mod [\Z G,\Z G]$. Then $x^p \equiv v^p \equiv y \mod ([\Z G,\Z G]+p\Z G)$, by Lemma~\ref{pthPowerCommutator}.
	Therefore using bar notation for images in $\F_p G$ we deduce that
	$\overline{x}^p \equiv \overline{y} \mod [\F_pG,\F_pG]$ and hence $\varepsilon_g(\overline{x}^p)=\varepsilon_g(\overline{y})$ for every $g\in G$. In particular $\varepsilon_y(\overline{x^p})=\varepsilon_y(\overline{y})=1$ and hence $x^p$ and $y$ are conjugate in $G$, as desired.
\end{proof}

\section{Double action}

In this section we rewrite the Zassenhaus Problems in terms of isomorphisms between certain modules.

In the remainder $G$ and $H$ are finite groups and $R$ is a commutative ring.
Fix a group homomorphism
$$\alpha:H\rightarrow \U(RG).$$
Then we define a left $R(H\times G)$-module $R[\alpha]$ as follows:
As an $R$-module $R[\alpha]=RG$ and the multiplication by elements of $H\times G$ is given by the following formula:
\begin{equation}\label{MultiplicationDouble}
(h,g)v=\alpha(h)vg^{-1}, \quad (h\in H, g\in G, v\in RG).
\end{equation}

We consider $G$ as a subgroup of $H\times G$ via the map $g\mapsto (1,g)$.
Let $\alpha,\beta:H\rightarrow \U(RG)$ be two group homomorphism.
Then $R[\alpha]$ and $R[\beta]$ are isomorphic as left $RG$-modules and every isomorphism  between them as $RG$-module is given as follows
\begin{eqnarray*}
	\rho_u:RG & \rightarrow & RG \\ x & \mapsto & ux
\end{eqnarray*}
for some $u\in \U(RG)$.
Moreover $\rho_u$ is an isomorphism of $R(H\times G)$-modules if and only if $\beta(h)=u\alpha(h)u^{-1}$ for every $h\in H$. This proves the following:

\begin{proposition}\label{HomModuleIsomorphic}
	Let $\alpha,\beta:H\rightarrow \U(RG)$ be group homomorphisms.
	Then $R[\alpha]\cong R[\beta]$ if and only if there is $u\in \U(RG)$ such that $\beta(h)=u \alpha(h)u^{-1}$ for every $h\in H$.
\end{proposition}

The connection of Proposition~\ref{HomModuleIsomorphic} with the Zassenhaus Problems is now clear:

\begin{corollary}\label{ZPModuleRewritting}
	The following are equivalent for a group homomorphism $\alpha:H\rightarrow V(RG)$:
	\begin{enumerate}
		\item There is $u\in \U(RG)$ and a group homomorphism $\sigma:H\rightarrow G$ such that $\alpha(h)=u^{-1} \sigma(h) u$ for every $h\in H$.
		\item $\alpha(H)$ is conjugate within $\U(RG)$ to a subgroup of $G$
		\item $R[\alpha]\cong R[\sigma]$ for some group homomorphism $\sigma:H\rightarrow G$.
	\end{enumerate}
	Furthermore, if $R$ is a field of characteristic zero then the above conditions are equivalent to the following:
	\begin{itemize}
		\item[(4)] The character afforded by $R[\alpha]$ is equal to the character afforded by $R[\sigma]$ for some group homomorphism $\sigma:H\rightarrow G$.
	\end{itemize}
\end{corollary}

Corollary~\ref{ZPModuleRewritting} suggests to calculate the character $\chi_{\alpha}$ of $H\times G$ afforded by the module $R[\alpha]$. Using $G$ as a basis of $R[\alpha]$ as $R$-module one easily  obtains the following
\begin{equation}\label{CharacterRalpha}
\chi_{\alpha}(h,g) = |C_G(g)| \varepsilon_g(\alpha(h)).
\end{equation}

The following proposition will be proved in Section~\ref{SectionSpectrum}

\begin{proposition}[Hertweck] \label{OrderPA0}
	Let $u$ be a torsion element of $V(\Z G)$ and let $g\in G$. If $|g|$ does not divide $|u|$ then $\varepsilon_g(u)=0$.
\end{proposition}

Let $\Cl(G)$ denote the set of conjugacy classes of $G$.
If $C\in \Cl(G)$ and $g\in C$ then, by definition, the order of $C$ is the order of $g$ and for every integer $k$, $C^k$ denotes the conjugacy class of $C$ in $G$ containing $g^k$.
Let $\Cl_m(G)$ denote the set of conjugacy classes of $G$ of order dividing $m$.

\begin{lemma}\label{PAPowers}
Let $u$ be a torsion element of order $n$ in $V(\Z G)$, let $k$ be a positive integer coprime with $n$ and let $C$ be a conjugacy class in $G$.
Then
\begin{equation}\label{PAP}
\varepsilon_C(u^k)=\sum_{\substack{D\in \Cl(G) \\ D^k = C}} \varepsilon_D(u).
\end{equation}
\end{lemma}

\begin{proof}
	Let $m$ denote the order of $C$.
	If $m\nmid n$ then the order of every $D\in \Cl(G)$ with $D^k=C$ does not divide $n$ and hence, by Proposition~\ref{OrderPA0}, we have $\pa{u^k}{C}=\pa{u}{D}=0$ for every such $D$. Then \eqref{PAP} holds.
	
	Suppose otherwise that $m\mid n$ and let $l$ be an integer such that $kl\equiv 1 \mod n$.
	Then $C^l$ is the unique element $D$ of $\Cl(G)$ with $D^k=C$.
	Thus we have to prove that $\pa{u^k}{C}=\pa{u}{C^l}$.
	Let $\alpha:\GEN{u}\rightarrow V(\Z G)$ denote the inclusion map.
	The representation $\rho$ of $\GEN{u}\times G$ associated to the module $\Z[\alpha]$ has degree $|G|$ and affords the character $\chi=\chi_{\alpha}$.
	Let $g\in C$. By assumption the order of $(u^k,g)$ is $n$.
	Let $\zeta_n$ denote a complex primitive $n$-th root of unity.
	Then $\rho(u^k,g)$ is conjugate to $\diag(\zeta_n^{i_1},\dots,\zeta_n^{i_{|G|}})$ for some integers $i_1,\dots,i_{|G|}$ and $\rho(u,g^l)$ is conjugate to $\diag(\zeta_n^{li_1},\dots,\zeta_n^{li_{|G|}})$.
	As $\gcd(l,n)=1$, there is an automorphism $\sigma$ of $\Q(\zeta_n)$ given by $\sigma(\zeta_n)=\zeta_n^l$.
	Moreover, $\chi(u^k,g)\in \Z$, by \eqref{CharacterRalpha}.
	Then $\chi(u^k,g)=\sigma(\chi(u^k,g))=\sum_{j=1}^{|G|} \zeta_n^{li_j}=\chi(u,g^l)$.
	Applying again \eqref{CharacterRalpha} and  $C_G(g)=C_G(g^l)$ we have  $\pa{u^k}{C}=\pa{u^k}{g} = \pa{u}{g^l}=\pa{u}{C^l}$, as desired.
\end{proof}

%\begin{lemma}\label{PAPowers}
%If $u$ is a torsion element of $V(\Z G)$, $g\in G$ and $k$ is coprime with the orders of $u$  and $g$ then $\varepsilon_{g^k}(u^k)=\varepsilon_g(u)$.
%\end{lemma}
%
%
%\begin{proof}
%	Let $n$ be the least common multiple of the orders of $u$ and $g$ and let $\alpha:\GEN{u}\rightarrow V(\Z G)$ be the inclusion map and $\chi=\chi_{\alpha}$. Let $\rho$ be a representation affording the character $\chi$.
%	By the assumption, the exponent of $\GEN{(u,g)}$ is $n$ and there is an automorphism $\sigma$ of $\Q(\zeta_n)$ mapping $\zeta_n$ to $\zeta_n^k$.	
%	Moreover, the eigenvalues of $\rho(u^k,g^k)$ are the $k$-th powers of the eigenvalues of $\rho(u,g)$.
%	Therefore $\sigma(\chi(u,g))=\chi(u^k,g^k)$.
%	However, by \eqref{CharacterRalpha}, $\chi(u,g)\in \Z$ and hence $\chi(u,g)=\chi(u^k,g^k)$.
%	Using \eqref{CharacterRalpha} again we obtain $\varepsilon_g(u)=\frac{\chi_{\alpha}(u,g)}{|C_G(g)|} = \frac{\chi_{\alpha}(u^k,g^k)}{|C_G(g^k)|} = \varepsilon_{g^k}(u^k)$.
%\end{proof}

Using Lemma~\ref{PAPowers} and Theorem~\ref{MRSW} one can obtain the following simplified version of the latter.

\begin{corollary}\label{ZCPA}
	Let $u$ be an element of $V(\Z G)$ of order $n$. Then the following are equivalent.
	\begin{enumerate}
		\item $u$ is rationally conjugate to an element of $G$.
		\item For every $d\mid n$, there is $g_d\in G$ with $\varepsilon_g(u^d)=0$ for every $g\in G\setminus g_d^G$.
		\item $\varepsilon_g(u^d)\ge 0$, for every $d\mid n$ and $g\in G$.
	\end{enumerate}	
\end{corollary}

\begin{proof}
By Theorem~\ref{MRSW}, it is enough to show that if (3) holds then $\varepsilon_C(u^i)\ge 0$ for every positive integer $i$ and every $C\in \Cl(G)$. Indeed, suppose that (3) holds, let $i$ be a positive integer and let $d=\gcd(i,n)$ and $k=\frac{i}{d}$. Then $\frac{n}{d} = |u^d|$ and $\gcd(k,\frac{n}{d})=1$. Then, by Lemma~\ref{PAPowers}, we have $\pa{u^i}{C} = \sum_{\substack{D\in \Cl(G) \\ D^k = C}} \pa{u^d}{D}\ge 0$.
\end{proof}

\begin{example}\label{ExamplePrimeOrder}{\rm
Combining the Berman-Higman Theorem and Proposition~\ref{OrderPA0} we deduce that if the order of $u$ is prime, say $p$, then  $\varepsilon_g(u)=0$ for every $g\in G$ of order $\ne p$.
If all the elements of order $p$ form a conjugacy class of $G$ then $u$ satisfies the conditions of Theorem~\ref{MRSW} and thus $u$ is conjugate in $\Q G$ of
an element of $G$.
For example this holds $G=S_3$ and any $p$; for $G=S_5$ and $p=3$ or $5$; and for $G=\A_5$ and $p=2$ or $3$.
However this is not valid for $G$ either $S_4$ or $S_5$ and $p=2$; nor for $G=\A_5$ and $p=5$. In the first case there are two conjugacy classes of elements of
order $2$, one containing $(1,2)$ and another one containing $(1,2)(3,4)$.
In the second case, there are two conjugacy classes of elements of order $5$ in $\A_5$.
}\end{example}

\section{The HeLP Method}

Let $\zeta_n$ denote a complex primitive $n$-th root of unity and set $F=\Q(\zeta_n)$.
Then every automorphism of $F$ is given by $\sigma_i(\zeta_n)=\zeta_n^i$ with $i\in \U(\Z/n\Z)$, i.e. $i$ is an integer coprime with $n$.
Consider the Vandermonde matrix
$$V=V(1,\zeta_n,\zeta_n^2,\dots,\zeta_n^{n-1}) = \pmatriz{{ccccc} 1 & 1 & 1& \dots & 1 \\ 1 & \zeta_n & \zeta_n^2 & \dots & \zeta_n^{n-1} \\
	1 & \zeta_n^2 & \zeta_n^{2^2} & \dots & \zeta_n^{2(n-1)} \\
	\dots & \dots & \dots & \dots & \dots \\
	1 & \zeta_n^{(n-1)} & \zeta_n^{2(n-1)} & \dots & \zeta_n^{(n-1)^2}}$$
and its complex conjugate
$$\overline{V}=V(1,\overline{\zeta_n},\overline{\zeta_n}^2,\dots,\overline{\zeta_n}^{n-1}) = V(1,\zeta_n^{-1},\zeta_n^{-2},\dots,\zeta_n^{1-n}).$$
The $(i,j)$-th entry of $V\overline{V}$ is
$$\sum_{k=0}^{n-1} \zeta_n^{k(i-j)} = \begin{cases} n, & \text{if } i=j; \\ 0, & \text{otherwise}.\end{cases}.$$
Therefore
$$V^{-1}=\frac{1}{n}\overline{V}.$$

Let $U\in M_k(\C)$ with $U^n=1$.
Then the eigenvalues of $U$ are of the form $\zeta_n^i$ with $i=0,1,\dots,n-1$.
Let $\mu_i$ denote the multiplicity of $\zeta_n^i$ as eigenvalue of $U$, i.e. $U$ is conjugate in $M_k(\C)$ to a diagonal matrix where each $\zeta_n^i$ appears $\mu_i$ times in the diagonal. We denote this diagonal matrix as
	$$\diag(1\times \mu_0,\zeta_n\times \mu_1,\dots,\zeta_n^{n-1}\times \mu_{n-1}).$$
Then $U^j$ is conjugate in $M_k(\C)$ to $\diag(1\times \mu_0,\zeta_n^j\times \mu_1,\dots,\zeta_n^{j(n-1)}\times \mu_{n-1})$.
Therefore
\begin{equation}\label{TraceMultiplicities}
\tr(U^j) = \mu_0+\mu_1\zeta_n^j+\mu_2\zeta_n^{2j}+\dots+\mu_{n-1}\zeta_n^{(n-1)j},
\end{equation}
for all $j$, or equivalently
$$\pmatriz{{c} \tr(U^0) \\ \tr(U) \\ \tr(U^2) \\ \vdots \\ \tr(U^{n-1})} = V   \pmatriz{{c} \mu_0 \\ \mu_1 \\ \mu_2 \\ \vdots \\ \mu_{n-1}}.$$
Thus
$$ \pmatriz{{c} \mu_0 \\ \mu_1 \\ \mu_2 \\ \vdots \\ \mu_{n-1}}= \frac{1}{n}\overline{V} \pmatriz{{c} \tr(U^0) \\ \tr(U) \\ \tr(U^2) \\ \vdots \\
	\tr(U^{n-1})},$$
or equivalently
\begin{equation}\label{MultiplicitiesTrace}
\mu_i = \frac{1}{n}\sum_{j=0}^{n-1} \tr(U^j) \zeta_n^{-ij}.
\end{equation}
If $d=\gcd(j,n)$ then $\sigma_{\frac{j}{d}}\in \Gal(\Q(\zeta_n^d)/\Q)$ and $\zeta_n^{-ij}=\sigma_{\frac{j}{d}}(\zeta_n^{-id})$.
Combining this with (\ref{TraceMultiplicities}), we deduce that $\tr(U^j)=\sigma_{\frac{j}{d}}(\tr(U^d))$ and hence, grouping the summands in the right side of (\ref{MultiplicitiesTrace}) with the same greatest common divisor with $n$, we have 
    \begin{equation}\label{MultiplicitiesDivisors}
    \mu_i = \frac{1}{n} \sum_{d\mid n} \Tr_{\Q(\zeta_n^d)/\Q}(\tr(U^d) \zeta_n^{-id}).
    \end{equation}

Suppose now that $u$ is an element of order $n$ of $\U(\C G)$ and $\rho$ is a representation of $G$ affording the character $\chi$.
Applying (\ref{MultiplicitiesDivisors}) to $U=\rho(u)$ we deduce that the multiplicity of $\zeta_n^i$ as an eigenvalue of $\rho(u)$ is
$$\mu(\zeta_n^i,u,\chi) := \frac{1}{n} \sum_{d\mid n} \Tr_{\Q(\zeta_n^d)/\Q}(\chi(u^d) \zeta_n^{-id}).$$

We are going to use that $\chi$ is constant on conjugacy classes to consider $\chi$ as a map defined on $\Cl(G)$, i.e. we denote $\chi(C)=\chi(g)$ whenever $C=g^G$ with $g\in G$.
By the linearity of $\chi$, for every $a\in \C G$ we have
$$\chi(a) = \sum_{C\in \Cl(G)} \varepsilon_C(a) \chi(C).$$
Therefore
\begin{equation}\label{Multiplicidad}
\mu(\zeta_n^i,u,\chi) = \frac{1}{n} \sum_{d\mid n} \sum_{C\in \Cl(G)} \varepsilon_C(u^d) \Tr_{\Q(\zeta_n^d)/\Q}(\chi(C) \zeta_n^{-id}).
\end{equation}
Observe that $\Tr_{\Q(\zeta_n^d)/\Q}(\chi(C) \zeta_n^{-id})$ makes sense in summands where $\varepsilon_C(u^d)\ne 0$. This is a consequence of Proposition~\ref{OrderPA0} because in that case the order of $C$ divides $\frac{n}{d}$ and hence $\chi(C)\in \Q(\zeta_n^d)$.
Thus, in the previous formula it is enough to run on the conjugacy classes $C$ in $\Cl_{\frac{n}{d}}(G)$.
As each $\mu(\zeta_n^i,u,\chi)$ is a non-negative integer we deduce:

\begin{proposition}[Luthar-Passi \cite{LutharPassi1989}]\label{LutharPassi}
	Let $u\in \U(\Z G)$ with $u^n=1$ and let $\chi$ be an ordinary character of $G$. Then
	\begin{equation}\label{LPEquation}
	\frac{1}{n} \sum_{d\mid n} \sum_{C\in \Cl_{\frac{n}{d}}(G)} \varepsilon_C(u^d) \Tr_{\Q(\zeta_n^d)/\Q}(\chi(C) \zeta_n^{-id}) \in \Z^{\ge 0}.
	\end{equation}
\end{proposition}

The Luthar-Passi Method uses \eqref{LPEquation} to limit the possible partial augmentations of powers of $u$ for an element of order $n$. More precisely, suppose that we want to prove the Zassenhaus Conjecture for a group $G$.
By the Cohn-Livingstone Theorem (Proposition~\ref{CohnLivingstone}) we know that if $V(\Z G)$ has an element of order $n$ then $n$ divides the exponent of $G$.
So we first calculate the exponent of $G$ and we consider all the possible divisors $n$ of this exponent.
For each of these $n$ we calculate all the tuples $(\varepsilon_{d,C})_{d\mid n, C\in \Cl_{\frac{n}{d}}(G)}$ of integers satisfying $\sum_{C\in \Cl_{\frac{n}{d}}(G)} \varepsilon_{d,C}=1$ for every $d\mid n$ and the following conditions:
$$\frac{1}{n} \sum_{d\mid n} \sum_{C\in \Cl_{\frac{n}{d}}(G)} \varepsilon_{d,C} \Tr_{\Q(\zeta_n^d)/\Q}(\chi(g) \zeta_n^{-id}) \in \Z^{\ge 0}.$$
We consider the $\varepsilon_{d,C}$ as the partial augmentations $\varepsilon_C(u^d)$ for a unit $u$ of order $n$.
By Corollary~\ref{ZCPA}, if all the tuples satisfying these conditions are formed by non-negative integers for all the possible values of $n$ then (ZC1) holds for $G$. In that case we say that the Luthar-Passi Method gives a positive solution of (ZC1) for $G$.

\begin{example}\label{A5Order5}{\rm
		Luthar and Passi proved the Zassenhaus Conjecture for $\A_5$ \cite{LutharPassi1989}.
		Here we show how they proved that every unit of prime order in $V(\Z \A_5)$ is rationally conjugate to an element of $\A_5$.
		Let $u$ be an element of order $p$ of $V(\Z \A_5)$, with $p$ prime.
		By the Cohn-Livingstone Theorem $\A_5$ has an element of order $p$ and hence $p$ is either $2$, $3$ or $5$.
		We have already seen in Example~\ref{ExamplePrimeOrder} that if $p=2$ or $p=3$ then $u$ is rationally conjugate to an element of $\A_5$.
		Suppose that $p=5$. 		
		The alternating group $\A_5$ has two conjugacy classes of elements of order $5$ which we are going to denote $5a$ and $5b$.
		Let $\varepsilon_1$ and $\varepsilon_2$ denote the partial augmentations of $u$ at representatives of $5a$ and $5b$, respectively.
		By the Berman-Higman Theorem and Proposition~\ref{OrderPA0}, all the partial augmentations of $u$ other than $\varepsilon_1$ and $\varepsilon_2$ vanish.
		By Theorem~\ref{MRSW}, to prove that $u$ is conjugate in $\Q \A_5$ to an element of $\A_5$ we need to show that
		$\varepsilon_1,\varepsilon_2\ge 0$.
		$\A_5$ has an irreducible character $\chi$ of degree $3$ with $\chi(5a)=-\zeta_5-\zeta_5^{-1}$ and $\chi(5b)=-\zeta_5^2-\zeta_5^{-2}$.
		Applying Proposition~\ref{LutharPassi} we have
		\begin{equation}\label{LPA5}
		\frac{1}{5}\left(\varepsilon_1 \Tr_{\Q(\zeta_5)/\Q}(-(\zeta_5+\zeta_5^{-1})\zeta_5^{-i})+\epsilon_2
		\Tr_{\Q(\zeta_5)/\Q}(-(\zeta_5^2+\zeta_5^{-2})\zeta_5^{-i})+3\right) \in \Z^{\ge 0}.
		\end{equation}
		Moreover
		$$\Tr_{\Q(\zeta_5)/\Q}(-(\zeta_5^j+\zeta_5^{-j})\zeta_5^{-i})=\begin{cases} -3, & \text{if } i\equiv \pm j \mod 5; \\ 2, & \text{otherwise}.\end{cases}$$
		Applying (\ref{LPA5}), for $i=1$ and $i=2$ gives $\epsilon_2=1-\varepsilon_1,\varepsilon_1\in \Z^{\ge 0}$, as desired.
		We conclude that $u$ is conjugate in $\Q \A_5$ to an element of $\A_5$.
		
Luthar and Passi used the same method to prove that $V(\Z \A_5)$ has no elements of order $6$, $10$ or $15$ by showing that there are no integers $\varepsilon_{d,C}$ satisfying the restrictions of the Luthar-Passi Method. By the Cohn-Livingstone Theorem (Theorem~\ref{CohnLivingstone}) the order of every torsion element of $V(\Z \A_5)$ is a divisor of $30$ and, as there are no elements of orders $6$, $10$ or $15$, then every order is either $2$, $3$ or $5$. Thus (ZC1) holds for $\A_5$.
}\end{example}

The last paragraph of the previous example shows how one can use the Luthar-Passi Method to obtain positive solutions for the Spectrum Problem or the Prime Graph Question.

Hertweck extended \eqref{Multiplicidad} to Brauer characters.
We recall the definition of Brauer characters.
Let $p$ be a prime integer.
Let $G_{p'}$ denote the set formed by the $p$-\emph{regular} elements of $G$, i.e. those of order coprime with $p$.
Let $m$ be the least common multiple of the elements of $G_{p'}$ and fix $\zeta_m$ a complex primitive $m$-th root of unity and $\xi_m$ a primitive $m$-th root of unity in a field $F$ of characteristic $p$.
Let $\rho$ be an $F$-representation of $G$ and let $g\in G_{p'}$.
Then $\rho(g)$ is conjugate to $\diag(\xi_m^{i_1},\dots,\xi_m^{i_k})$ for some integers $i_1,\dots,i_k$. Thus the character afforded by $\rho$ maps $g$ to $\xi_m^{i_1}+\dots+\xi_m^{i_k}$.
By definition, the Brauer character afforded by $\rho$ is the map $\chi:G_{p'}\rightarrow \C$ associating $g$ with $\zeta_m^{i_1}+\dots+\zeta_m^{i_k}$.
Composing $\rho$ with the natural projection $\Z G\rightarrow \F_p G\subseteq FG$ we obtain a ring homomorphism $\rho:\Z G\rightarrow M_n(F)$.
Then \eqref{Multiplicidad} gives the multiplicity of $\xi_n^i$ as an eigenvalue of $\rho(u)$ \cite{HertweckBrauer}.
This provides more constrains to the possible partial augmentations of a $p$-regular units. 
This has been used to obtain positive solutions for (ZC1) for cases where the equations provided by ordinary characters are not sufficient.

\section{The Spectrum Problem holds for solvable groups}\label{SectionSpectrum}

In this section we prove Proposition~\ref{OrderPA0} and that the Spectrum Problem holds for solvable groups.
Both are results of Hertweck. 
For the proofs one uses the following results.

\begin{theorem}\label{ProjectiveCyclicpGroup}\cite[Chapter~2]{Alperin1986}
	Let $C$ be a finite cyclic $p$-group with generator $c$ and let $F$ be a field of characteristic $p$.
	Let $M$ be a finite dimensional $FC$-module of dimension $k$ over $F$.
	Then $M$ is indecomposable if and only if $1\le k\le |C|$ and the Jordan form of $\rho(c)$ is an elementary Jordan matrix. Moreover, in that case $M$ is projective if and only if $k=|C|$.
\end{theorem}

Observe that if $M$ satisfies the conditions of Theorem~\ref{ProjectiveCyclicpGroup} then the order of the Jordan form $J_k(a)$ of $\rho(c)$ is a power of $p$. This implies that $a$ is a root of unity of order a power of $p$ in $F$. As $F$ has characteristic $p$ this implies that $a=1$. So $M$ is indecomposable if and only if $\rho(c)$ is conjugate to $J_k(1)$ with $1\le k\le |C|$.
Combining this with the last statement of Theorem~\ref{ProjectiveCyclicpGroup} we deduce that $FC$ has a unique projective indecomposable $FC$-module and it has dimension $|C|$. As $FC$ is projective of dimension $|C|$, it follows that it is the unique indecomposable projective $FC$-module.

Recall that a Dedekind domain is a noetherian integrally closed commutative domain for which every non-zero prime ideal is maximal.

\begin{theorem}\label{ProjectiveCharacter0} \cite[(32.15)]{CurtisReiner1981}
Let $R$ be a Dedekind domain of characteristic $0$.	
	If $_{RG}M$ is projective, $\chi$ is the character afforded by $M$ and $g\in G$ is such
	that $|g|$ is not invertible in $R$ then $\chi(g)=0$.
\end{theorem}

\begin{theorem}\label{ProjectiveSubgroup} \cite[Theorem~9.1]{BensonGoodearl}
	Let $R$ be a Dedekind domain of characteristic $0$ and let $M$ be an $RG$-module.
	If $H$ is a subgroup of $G$ then $_{RG}M$ is projective if and only if $_{RH}M$ is
	projective and $R/Q\otimes_R M$ is
	projective as $(R/Q)G$-module for every maximal ideal $Q$ of $R$ containing $[G:H]$.
\end{theorem}

\begin{lemma}\label{ProjectiveCoprime}
	Let $R$ be a ring, let $M$ be a left $RG$-module and let $H$ be a subgroup of $G$ such that $[G:H]$ is invertible in $R$.
	If $M$ is projective as $RH$-module then $M$ is projective as $RG$-module.
\end{lemma}

\begin{proof}
	Suppose that $M$ is projective as $RH$-module and let $\alpha:N\rightarrow M$ be a surjective homomorphism of $RG$-modules. We have to show that $\alpha$ splits.
	As $M$ is projective as $RH$-module, there is a homomorphism $\beta:M\rightarrow N$ of $RH$-modules such that $\alpha\beta=1_M$.
	Fix a right transversal of $H$ in $G$. Then for every $g\in G$ and $t\in T$ there are unique $s(t,g)\in T$ and $h(t,g)\in H$ such that $tg=h(t,g)s(t,g)$. Moreover for every $g\in G$, the map $t\mapsto s(t,g)$ is a permutation of the elements of $T$ (check it!).
	Let $\overline{\beta}:M\rightarrow N$ be given by
	$$\overline{\beta}(m) = \frac{1}{[G:H]} \sum_{t\in T} t^{-1} \beta(tm) \quad (m\in M).$$
	Then $\overline{\beta}$ is a homomorphism of $RG$-modules because if $g\in G$ and $m\in M$ then
	\begin{eqnarray*}
		\overline{\beta}(gm) &=& \frac{1}{[G:H]} \sum_{t\in T} t^{-1} \beta(tgm) =
		\frac{1}{[G:H]} \sum_{t\in T} t^{-1} \beta(h(t,g)s(t,g)m) \\
		&=&
		\frac{1}{[G:H]} \sum_{t\in T} t^{-1}h(t,g) \beta(s(t,g)m) = g\frac{1}{[G:H]} \sum_{t\in T} s(t,g)^{-1} \beta(s(t,g)m) \\
		&=&
		g\frac{1}{[G:H]} \sum_{t\in T} t^{-1} \beta(tm) = g\overline{\beta}(m).
	\end{eqnarray*}
	Moreover, $\alpha \overline{\beta}(m)= \frac{1}{[G:H]} \sum_{t\in T} t^{-1} \alpha \beta(tm) = m$ as $\alpha$ is a homomorphism of $RG$-modules and $\alpha\beta=1_M$.
\end{proof}

\begin{lemma}[Hertweck \cite{Hertweck2006}]\label{ProjectiveCyclic}
	Let $p$ be a prime integer and let $F$ be a field of characteristic $p$.
	Let $C$ be a non-trivial cyclic $p$-group and let $P$ be the subgroup of $C$ of order $p$.
	Let $M$ be an $FG$-module which is finitely generated over $F$.
	Then $M_{FC}$ is projective if and only if $M_{FP}$ is projective.
\end{lemma}
\begin{proof}
	Using that $FC_{FP}$ is free, it follows easily that if $M_{FC}$ is projective, then so is $M_{FP}$.
	
To prove the converse we may assume that $M_{FC}$ is indecomposable and $|C|>p$ and fix a generator $c$ of $C$.
By Theorem~\ref{ProjectiveCyclicpGroup} and the comments afterwards, the matrix expression of the multiplication by $c$ map in a suitable basis $v_1,\dots,v_k$ of  $M_K$ is a Jordan matrix
	$$\rho(c)=J_k(1)=\pmatriz{{cccc} 1 \\ 1 & 1  \\ & \ddots & \ddots \\ & & 1 & 1}\in M_k(F)$$
with $1\le k \le |C|$. Moreover, $M_{FC}$ is projective if and only if $k=|C|$.
	
	Suppose that $M_{FP}$ is projective.
	We want to show that $M_{FC}$ is projective and this is equivalent to showing that $k=|C|$ by the previous paragraph.
	Let $q=\frac{|C|}{p}$. Then $P=\GEN{g^q}$ and $\rho(g^q)=J_k(1)^q$.
	Therefore
	$$g^q v_i =\begin{cases} v_i+v_{i+q}, & \text{if } i+q\le k; \\ v_i, & \text{otherwise}.\end{cases}$$
	As $M_{FP}$ is projective, the action of $P$ on $M$ is non-trivial and therefore $J_k(1)^q\ne I$. Therefore, $k>q$.
	Write $k-q+1=sq+t$ with $s$ and $t$ non-negative integers and $t<q$. 
	Let $I=\{t+iq:i=0,\dots,s\}$, $J=\{1,\dots,k\}\setminus I$, 
	$M_I=\sum_{i\in I} Fv_i$ and $M_J=\sum_{j\in J} Fv_j$. 
	Clearly, $M=M_I\oplus M_J$ and the expression above of $g^q v_i$ implies that $M_I$ and $M_J$ are submodules of $FP$.
	As $M_{FP}$ is projective, so are $M_I$ and $M_J$ and hence the dimension of both is a multiple of the dimension of the unique projective indecomposable $FP$-module.
	As this dimension is $p$ we deduce that $p\mid s+1$ and $p\mid k$. Thus $1\equiv t \mod p$ and hence $|C|\mid (s+1)q=k+1-t\le k$. Thus $k=|C|$ as desired.
\end{proof}

We are ready for the

\begin{proofof} \textbf{Proposition~\ref{OrderPA0}}.
	Suppose that $|g|$ does not divide $|u|$. Then there is a prime integer $p$ and a positive integer $n$ such that $p^n$ divides $|g|$ but $p^n$ does not divides $|u|$.
	Let $R=\Z_{(p)}$ be the localization of $\Z$ at $(p)$ and let $F=R/pR\cong \F_p$.
	Consider the inclusion $\alpha:\GEN{u}\rightarrow V(\Z G)\subseteq V(RG)$ and let $M=R[\alpha]$.
	Let $C=\GEN{(u,g)}$, let $P$ be the Sylow $p$-subgroup of $C$ and let $Q$ be the subgroup of $P$ of order $p$.
	By the assumption on the orders of $u$ and $g$, $Q=\GEN{(1,k)}$ with $\GEN{k}$ the subgroup of order $p$ of $\GEN{g}$.
	Then $_{FQ}(F\otimes_R M) \cong {_{F\GEN{k}}(F\otimes_R M)} \cong {_{F\GEN{k}}FG}= F \GEN{k}^{[G:\GEN{k}]}$, which is free and hence projective.
	Then $_{FP}F\otimes_R M$ is projective, by Lemma~\ref{ProjectiveCyclic} and thus $_{RP} M$ is projective by Theorem~\ref{ProjectiveSubgroup} (applied with $G=P$ and $H=1$). 
	As $[C:P]$ is invertible in $R$, by Theorem~\ref{ProjectiveCoprime}, we deduce that $_{RC}M$ is projective.
	Moreover, $|(u,g)|$ is divisible by $p$ and hence it is not invertible in $R$. Then $\chi((u,g))=0$, by
	Theorem~\ref{ProjectiveCharacter0}.
	Finally, $\varepsilon_g(u)=\frac{\chi((u,g))}{|C_G(g)|}=0$, by \eqref{CharacterRalpha}.
\end{proofof}

Recall that if $g$ is an element of finite order in a group and $p$ is a prime integer then there are unique elements $h,k\in \GEN{g}$ such that $g=hk$ and $h$ is a $p$-element and $k$ is $p$-regular. Then $h$ and $k$ are called the $p$-part and $p'$-parts of $g$, respectively.

Basically the same proof of Proposition~\ref{OrderPA0}, now using Green's Theorem on Zeros of Characters \cite[(19.27)]{CurtisReiner1981}, gives the following:

\begin{proposition}[Hertweck \cite{HertweckOrdersSolvable}]\label{OrderPA02}
	Let $P$ be a normal $p$-subgroup of $G$. Let $u$ be a torsion unit of $V(\Z G)$ such that $|\aug_P(u)|<|u|$ and $g\in G$ such that the order of the
	$p$-part of $g$ is smaller than the order of the $p$-part of $u$. Then $\varepsilon_g(u)=0$.
\end{proposition}

\begin{proposition}[Hertweck \cite{HertweckOrdersSolvable}]\label{OrderPA0-2}
	If $G$ is solvable and $u$ is a torsion element of $V(\Z G)$ then $G$ has an element with the same order as $u$ such that $\varepsilon_g(u)\ne 0$.
\end{proposition}
\begin{proof}
Let $G$ be a solvable group and let $u$ be an element of order $n$ in $V(\Z G)$.
We argue by induction on the order of $G$. The result is clear if $G=1$.
So we suppose that $G\ne 1$ and the proposition holds for solvable groups of smaller order.
Since $G$ is solvable, it has a normal $p$-subgroup $P$ of $G$.
Use the bar reduction for reduction modulo $P$, i.e. $\overline{x}=\aug_P(x)$ for $x\in \C G$.

If $v$ is a torsion element of $V(\Z G)$ then $v^{|\overline{v}|}\in V(\Z G,P)$. Thus $v^{|\overline{v}|}$ is a $p$-element, by Lemma~\ref{VPPgroup}.
This shows that the $p'$-parts of $v$ and $\overline{v}$ have the same order.

By induction, there is $x\in G$ such that $|\overline{x}|=|\overline{u}|$ and  $\varepsilon_{\overline{x}}(\overline{u})\ne 0$.
The first, combined with the previous paragraph, implies that the $p'$-parts of $x$ and $u$ are equal have the same order.
Observe that $\varepsilon_{\overline{x}}(\overline{u})$ is the sum of the partial augmentations of the form $\varepsilon_{g}(u)$ with $\overline{g}$ conjugate
to $\overline{x}$. In particular, $\varepsilon_g(u) \ne 0$ for some $g\in G$ such that $\overline{g}$ is conjugate to $\overline{x}$ in $\overline{G}$.
Thus we may assume that $\varepsilon_x(u)\ne 0$. Then $|x|\mid |u|$, by Proposition~\ref{OrderPA0}.
If $|u|=|\overline{u}|$ then $|x|\mid |u| = |\overline{u}| = |\overline{x}| \mid |x|$ and hence $|x|=|u|$, as desired. Otherwise, by Proposition~\ref{OrderPA02} the $p$-parts of
$|x|$ and $|u|$ are equal. Thus $|x|=|u|$ and we are done.
\end{proof}

We finish with the result which justifies the title of this section.

\begin{theorem}\cite{Hertweck2008}
	The Spectrum Problem holds for solvable groups.
\end{theorem}

\bigskip

I would like to thank Andreas Bächle and Leo Margolis for reading a preliminary version of these notes and providing many great suggestions. I also would like to thank to Alexey Gordienko and Alonso Albaladejo for several remarks which helps to improve this notes.

\bibliographystyle{amsalpha}
\bibliography{ReferencesMSC}

\end{document}